\documentclass{amsart}
\usepackage{amsmath}
\usepackage{amsthm}
\usepackage{amsfonts}
\usepackage{amssymb}
\usepackage{bbm}
\usepackage{graphicx}
\usepackage{tikz}
\usepackage{subcaption}
\usepackage{natbib}
\usepackage{enumitem}
\usepackage{tikz-cd}
\usepackage[text={440pt,575pt},headheight=9pt,centering]{geometry}

\usepackage{hyperref}
\hypersetup{
	colorlinks=true,
	linkcolor=blue,
	filecolor=magenta,
	urlcolor=cyan,
	citecolor=teal
}

\allowdisplaybreaks

\DeclareFontFamily{U}{mathx}{\hyphenchar\font45}
\DeclareFontShape{U}{mathx}{m}{n}{
	<5> <6> <7> <8> <9> <10>
	<10.95> <12> <14.4> <17.28> <20.74> <24.88>
	mathx10
}{}
\DeclareSymbolFont{mathx}{U}{mathx}{m}{n}
\DeclareFontSubstitution{U}{mathx}{m}{n}
\DeclareMathAccent{\widecheck}{0}{mathx}{"71}

\newcommand{\x}{\mathbf{x}}
\newcommand{\y}{\mathbf{y}}
\newcommand{\z}{\mathbf{z}}
\newcommand{\w}{\mathbf{w}}
\newcommand{\X}{\mathbf{X}}
\newcommand{\Y}{\mathbf{Y}}
\newcommand{\W}{\mathbf{W}}

\newcommand{\RR}{\mathbb{R}}

\newcommand{\PP}{\mathbb{P}}
\newcommand{\EE}{\mathbb{E}}

\newcommand{\Det}{\operatorname{Det}}
\newcommand{\CM}{\operatorname{CM}}

\newcommand{\EMTPtwo}{$ \text{EMTP}_2 $ }

\newcommand{\ph}{\operatorname{ph}}
\newcommand{\indep}{\perp \!\!\! \perp}
 
\newcommand{\bs}{\boldsymbol}

\DeclareMathOperator*{\argmax}{\operatorname{arg max}}
\DeclareMathOperator{\diag}{diag}

\newtheorem{thm}{Theorem}[section]
\newtheorem{prop}[thm]{Proposition}

\newtheorem{lemma}[thm]{Lemma}

\newtheorem{algorithm}[thm]{Algorithm}

\theoremstyle{definition}

\newtheorem{ex}[thm]{Example}

\begin{document}
\title{Parametric and nonparametric symmetries in graphical models for extremes}

\author[F.~R\"ottger]{Frank R\"ottger$^1$}
\address[]{$^1$Research Center for Statistics\\Université de Genève\\ 1205 Geneva\\Switzerland}
\email{frank.roettger@unige.ch}

\author[J.~I.~Coons]{Jane Ivy Coons$^2$}
\address[]{$^2$St. John's College and Mathematical Institute, University of Oxford, United Kingdom}
\email{jane.coons@maths.ox.ac.uk}

\author[A.~Grosdos]{Alexandros Grosdos$^3$}
\address[]{$^3$Department of Mathematics, TUM School of Computation, Information and Technology, Technical University of Munich, Germany and  Institute of Mathematics, University of Augsburg, Germany }
\email{alex.grosdos@tum.de}
 
	\date{\today}

\keywords{Graphical models, Multivariate extremes, Parameter symmetries, Colored Graphs}

\makeatletter
\@namedef{subjclassname@2020}{\textup{2020} Mathematics Subject Classification}
\makeatother
\subjclass[2020]{Primary: 62H22 ; Secondary: 60G70, 62G32}
       \begin{abstract}
       Colored graphical models provide a parsimonious approach to modeling high-dimen- sional data by exploiting symmetries in the model parameters. In this work, we introduce the notion of coloring for extremal graphical models on multivariate Pareto distributions, a natural class of limiting distributions for threshold exceedances. Thanks to a stability property of the multivariate Pareto distributions, colored extremal tree models can be defined fully nonparametrically. For more general graphs, the parametric family of H\"usler--Reiss distributions allows for two alternative approaches to colored graphical models. We study both model classes and introduce statistical methodology for parameter estimation. It turns out that for H\"usler--Reiss tree models the different definitions of colored graphical models coincide. In addition, we show a general parametric description of extremal conditional independence statements for H\"usler--Reiss distributions. Finally, we demonstrate that our methodology outperforms existing approaches on a real data set.
      \end{abstract}

\maketitle
\section{Introduction}

Featuring heavily in applications,
graphical models are statistical models that capture conditional independencies among random variables 
by positing a structure based on a graph.
Modern statistical practice often involves inferring information from high-dimensional data
with relatively few samples, which makes the use of additional structure necessary.
A parsimonious approach to such a complexity reduction is to impose affine restrictions on graphical models. 
For parametric graphical models where each edge corresponds to a parameter, a natural choice are affine restrictions where parameters are assumed to be equal.
These models allow an elegant representation using colored graphs, which was introduced by \citet{HL2005,LauritzenColoured} for Gaussian graphical models. For example for a multivariate Gaussian that is Markov to the 4-cycle, the coloring in Figure~\ref{fig:4-cycle} would impose the affine restriction  $K_{14}=K_{23}$ on its concentration matrix $K$. 

Since their introduction, colored Gaussian graphical models (also referred to as graphical models with symmetries) have been heavily studied, both from a theoretical as well as an applied standpoint. For example, in \citet[Section 5]{FL2015}, the authors study the score matching estimator (SME) for Gaussian graphical models with and without symmetry. They propose the SME as a method for model selection in the colored case where the graphical lasso cannot be used. A necessary and sufficient algebraic-geometric condition for the existence of the maximum likelihood estimator (MLE) in colored graphical models was given in \cite{uhler2011}. Moreover, adding symmetries to the underlying graph can lower the minimal number of data points needed to guarantee the existence of the MLE (\cite{MRS21, uhler2018}). Coloring the edges and vertices in the graph reduces the number of parameters in the model which makes computations more tractable for large networks. For this reason, colored graphical models are popular for the analysis of gene coexpression networks (\cite{toh2002, vinciotti2016}).

In multivariate extreme value theory there is particular interest in sparse high-dimensional models as extreme observations are usually scarce.
The recent introduction of extremal conditional independence by \cite{EH2020} for multivariate Pareto distributions, which are the limit distribution for threshold exceedances, yields a parsimonious framework for the definition of sparse models in multivariate extremes via graphical modeling.
This has lead to a new line of research, see for example \cite{EV2020} for the introduction of nonparametric extremal tree models, \cite{ELV2021,LO2023} for approaches in structure learning, or \cite{REZ2021,RS2022} for the study of positive dependence in graphical extremes. A recent review article regarding sparsity in multivariate extremes is \cite{EI2021}.

In this paper we develop extremal colored graphical models to obtain further reduction in dimension and complexity.
For extremal tree models, which can be constructed from a vector of independent univariate random variables that are identified with the edges of a directed version of the underlying tree \citep{EV2020}, this allows to define nonparametric colored tree models (see Section~\ref{sec:def_col_tree}).
Thus a colored tree as in Figure~\ref{fig:claw} gives rise to a colored extremal tree model by imposing that within each edge color classes, all univariate random variables are identically distributed.

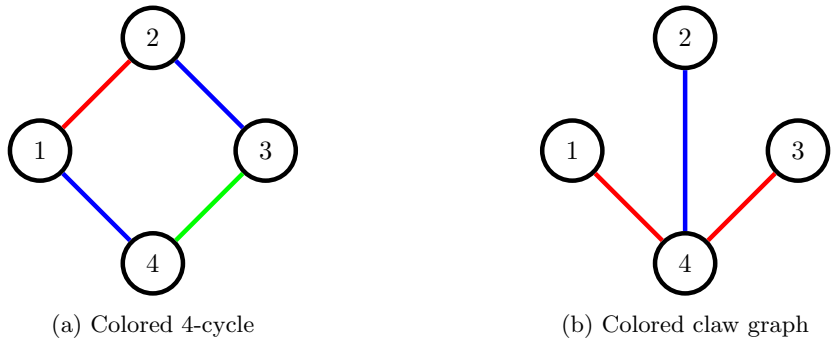
\begin{figure}
    \centering
    \begin{subfigure}{0.45\textwidth}
         \centering
   	\begin{tikzpicture}[line width=.6mm]
			\node[minimum size=8mm,shape=circle,draw=black] (1) at (-1.5,0) {1};
			\node[minimum size=8mm,shape=circle,draw=black] (2) at (0,1.5) {2};
			\node[minimum size=8mm,shape=circle,draw=black] (3) at (1.5,0) {3};
			\node[minimum size=8mm,shape=circle,draw=black] (4) at (0,-1.5) {4};
			\path[color=red] (1) edge (2);
			\path[color=blue] (2) edge (3);
			\path[color=green] (3) edge (4);
			\path[color=blue] (1) edge (4);
		\end{tikzpicture}
    \caption{Colored 4-cycle}\label{fig:4-cycle}
 \end{subfigure}  
    \begin{subfigure}{0.45\textwidth}
       \centering
   	\begin{tikzpicture}[line width=.6mm]
				\node[minimum size=8mm,shape=circle,draw=black] (1) at (-1.5,0) {1};
			\node[minimum size=8mm,shape=circle,draw=black] (2) at (0,1.5) {2};
			\node[minimum size=8mm,shape=circle,draw=black] (3) at (1.5,0) {3};
			\node[minimum size=8mm,shape=circle,draw=black] (4) at (0,-1.5) {4};
			\path[color=red] (1) edge (4);
			\path[color=blue] (2) edge (4);
			\path[color=red] (3) edge (4);			
		\end{tikzpicture}
    \caption{Colored claw graph}\label{fig:claw}
 \end{subfigure}    
 \caption{Examples for colored graphs with four nodes.}
\end{figure}

Among multivariate Pareto distributions, the parametric H\"usler--Reiss family \citep{HR1989} can be seen as an extremal analogue of the multivariate Gaussian.
A $d$-variate H\"usler--Reiss vector $\Y$ is parameterized by a variogram matrix $\Gamma$ that is in a one-to-one correspondence with a signed Laplacian matrix $\Theta$, that is a symmetric positive semidefinite matrix with vanishing row and column sums.
The matrices $\Gamma$ and $\Theta$ can be seen as the analogues of the covariance matrix $\Sigma$ and concentration matrix $K$ in the Gaussian case.
Indeed, \cite{HES2022} show that zeros in $\Theta$ encode the extremal graphical structure, i.e.~$Y_i$ is extremal conditionally independent of $Y_j$ given the remaining variables $\Y_{\setminus ij}$ ($Y_i\perp_eY_j|\Y_{\setminus ij}$) if and only if $\Theta_{ij}=0$.
In Section~\ref{sec:extrCI} we show that this parametric encoding extends to smaller conditioning sets. 
In Proposition~\ref{prop:CIviaGamma} we derive that for non-empty conditioning sets $C$ the equivalence
\[Y_i\perp_eY_j|\Y_{C}\Longleftrightarrow \det \CM(\Gamma)_{iC(d+1),jC(d+1)}=0, \]
holds, where $\CM(\Gamma)=\left(\begin{smallmatrix}
    -\Gamma/2 & \mathbf{1}\\
    -\mathbf{1}^T&0\\
\end{smallmatrix}\right)$
is the Cayley--Menger matrix \citep{devriendt2020effective}. This result can be used to translate statements involving extremal conditional independences into polynomial equations.
This allows us to employ techniques from algebraic statistics,
similarly to what has been previously done for Gaussian graphical models,
see for instance \cite{GeometryML12}.
We plan to unfold these results for the algebraically minded reader in a separate upcoming paper.

Through the relation between graphical structure and vanishing entries in $\Theta$, parametric restrictions allow the definition of colored H\"usler--Reiss graphical models. In Section~\ref{sec:def_RCON} we introduce H\"usler--Reiss restricted concentration (RCON) models, where edge colorings impose parametric constraints in the concentration matrix~$\Theta$.
For example the graph in Figure~\ref{fig:4-cycle} gives the affine constraint $\Theta_{14}=\Theta_{23}$.
In Section~\ref{sec:def_RVAR} we introduce H\"usler--Reiss restricted variogram (RVAR) models that capture parametric constraints in the variogram matrix~$\Gamma$.
Hereby the graph in Figure~\ref{fig:4-cycle} gives the affine constraint $\Gamma_{14}=\Gamma_{23}$.
Variogram models with parametric assumptions have been a standard tool in multivariate extremes for spatial applications, see for example in \cite{EMK2015,ADE2015}.
The entries in $\Gamma$ can be interpreted geometrically as squared Euclidean edge lengths in a simplex, see e.g.~\citet{devriendt2020effective}. Therefore RVAR constraints imply equilateral faces in the corresponding simplex.

When the underlying univariate random variables of a colored extremal tree model are chosen to be Gaussian, this yields a colored H\"usler--Reiss tree model.
We show in Proposition~\ref{prop:colHRtree} that these models have a direct link to RVAR and RCON models. 
In fact, we find that for H\"usler--Reiss tree models the definitions of RVAR, RCON and colored extremal tree models coincide.

In Section~\ref{sec:surrogate_lik} we develop statistical methodology for H\"usler--Reiss RCON and RVAR models.
For RCON models, we employ a surrogate maximum likelihood estimation approach. Hereby the original H\"usler--Reiss log-likelihood is substituted by a degenerate Gaussian log-likelihood \citep{REZ2021}. This relates to optimization problems in machine learning, in particular in graph signal processing, see e.g.~\cite{EHPO2017}.
We compute the score functions and information matrix of the degenerate Gaussian log-likelihood and implement a scoring algorithm to obtain a surrogate maximum likelihood estimate under RCON constraints.
For RVAR models, a similar surrogate maximum likelihood estimation approach is complicated by the highly nonlinear parameter restrictions imposed on the variogram matrix $\Gamma$, in addition to the graphical model constraints on $\Theta$. 
Therefore we follow the approach of \cite{RS2022} to use the mixed dual estimator of \cite{LZ2022} to incorporate constraints on both $\Gamma$ and $\Theta$. 
This leads to a two-step surrogate mixed dual estimator, where the second step is solved through a reciprocal scoring algorithm.

We illustrate our findings on real-world data. 
We test H\"usler--Reiss RCON and RVAR models on flight delay data from the United States \citep{HES2022} in Section~\ref{sec:flights}.
Splitting the data into training and validation data sets for four regional clusters in the United States, we compute an estimate for up to 25 edge color classes based upon an estimated graphical model.
We evaluate the H\"usler--Reiss log-likelihood on the validation dataset and observe a good fit for RCON estimates and an excellent fit for RVAR estimates in comparison to the underlying graphical model, while providing a drastic reduction in the number of parameters.

The \texttt{R} files with implementations of the scoring algorithms and the code for the application are available as supplementary material in a public \texttt{Github} repository \url{https://github.com/frank-unige/SymmetriesGraphicalExtremes}.

\section{Preliminaries}
\subsection{Graphical models in extremes}
Let $\X=(X_1,\ldots, X_d)$ be a $d$-variate random vector.
When interest is in the tail behavior of $\X$, the extremal correlation coefficient
\[\chi_{ij}:=\lim_{p\to 0}\PP(F_i(X_i)>1-p|F_j(X_j)>1-p)\]
given marginal distributions $F_i$ measures extremal dependence between $X_i$ and $X_j$ whenever the limit exists.
Hereby we differentiate between asymptotic dependence ($\chi_{ij}>0$) and asymptotic independence ($\chi_{ij}=0$).
In the setting of asymptotic dependence, the approach of threshold exceedances is particularly attractive for modeling (sparse) extremal dependence structures due to the recent introduction of extremal conditional independence and graphical models \citep{EH2020}.
As the marginal distributions $F_i$ do not influence the multivariate dependence structure, we can assume that all marginals $X_i$ follow a standard exponential distribution.
Under the assumption of multivariate regular variation \citep{res2008}, the limiting distribution of threshold exceedances
\begin{align}
\PP(\Y\le\y):=\lim_{u\to\infty}\PP(\X-u\mathbf{1}\le \y \ | \ \|\X\|_{\infty}>u)\label{eq:MPD}
\end{align}
is a \textit{multivariate (generalized) Pareto distribution} \citep{roo2006} for any $\X$ in the domain of attraction of $\Y$.
Hereby conditioning on the maximum of $\X$ to exceed a high threshold $u$ means that there should be an exceedance in at least one coordinate of $\X$. Given such an exceedance, we are interested in the distribution of the difference $\X-u\mathbf{1}$ for the limit $u\to \infty$.
By this construction, a multivariate Pareto distribution $\Y$ is supported on an L-shaped space $\mathcal{L}:=\{\x\in\RR^d:\|\x\|_{\infty}>0\}$, as we condition on at least one coordinate of $\X$ exceeding $u$.

When interest is in dependence modeling, we require the limit \eqref{eq:MPD} for marginal distributions of $\X$.
For a marginal $\X_I$ for some $I\subset [d]:=\{1,\ldots,d\}$, the resulting limit in \eqref{eq:MPD} is the $I$-th marginal of $\Y$ conditioned on exceeding zero in at least one dimension indexed by $I$, see \citet{HES2022}.
Although this is not the $I$-th marginal of $\Y$, we will nonetheless refer to this limit as $\Y_I$ for both convenience of notation and intelligibility. 

The support $\mathcal{L}$ of a multivariate Pareto distribution is not a product space, and therefore classical notions of stochastic independence are not applicable.
A parsimonious approach to this problem was proposed by \citet{EH2020}. Let
$\Y^k:=\Y|\{Y_k>0\},\; k\in[d]$
be conditional random vectors that are supported on half-spaces $\mathcal{L}^k:=\{\x\in\mathcal{L}:x_k>0\}$.
Note that the union of these half-spaces $\cup_{k\in[d]}\mathcal{L}^k=\mathcal{L}$ recovers the support of the multivariate Pareto vector $\Y$.
This structure permits a natural definition of extremal conditional independence, where for disjoint index sets $A,B,C\subset [d]$ we say that $\Y_A$ is extremal conditionally independent of $\Y_B$ given $\Y_C$ when for all $k\in[d]$ we have
that $\Y^k_A\indep\Y^k_B|\Y^k_C.$
We abbreviate this extremal conditional independence statement as $\Y_A\perp_e\Y_B|\Y_C$.

This notion of extremal conditional independence has appealing properties, see for example \citet[Section~4.2]{REZ2021}.
Moreover, it allows for a straightforward definition of extremal graphical models.
For an undirected graph $G=(V,\mathcal{E})$ with vertex set $V=[d]$ and edge set $\mathcal{E}\subset V\times V$ we call a multivariate Pareto vector $\Y$ an extremal graphical model when
$\Y_i\perp_e\Y_j|\Y_{\setminus ij}$ for all pairs $ij\not\in \mathcal{E}.$
The conditional independence statement above is also referred to as extremal pairwise Markov property. 
\cite{EH2020} show an extremal Hammersley--Clifford type theorem that provides the equivalence of the extremal pairwise Markov property with an extremal global Markov property and a Markov factorization property for decomposable graphs.

The conditional vectors $\Y^k$ allow a very useful stochastic representation as latent variable models, that is
$\Y^k\stackrel{d}{=}E\mathbf{1}+\W^k,$
where $E$ is a univariate standard exponential random variable.
The $k$th extremal function $\W^k$ is independent of $E$ and satisfies $W^k_k=0$.
This latent variable structure allows to construct multivariate Pareto distribution from a $(d-1)$-dimensional random vector $\W^k_{\setminus k}$ for any $k\in[d]$. Furthermore, $\W^k$ encodes extremal conditional independence when $k\in C$, such that
\begin{align*}
    \Y_i\perp_e\Y_j|\Y_{C}\Longleftrightarrow \W^k_i \indep \W^k_j |\W^k_C.
\end{align*}
This property is especially useful in graphical modeling, for example to construct extremal tree models as in the next section.

\subsection{Extremal tree models}\label{sec:ext_tree}
Multivariate Pareto distributions that are Markov to a tree show a particularly simple structure in the $k$th extremal function.
Given an undirected tree $T=(V,E)$, let $T^k=(V,E^k)$ be a directed tree rooted in $k$ by directing all edges away from $k$.
If $\Y$ is extremal Markov to $T$, the $k$th extremal function allows a representation
\begin{align}
    W_i^k=\sum_{e\in\ph(ki,T^k)} U_e,\label{eq:tree_construction}
\end{align}
for all $i\neq k$ where $\ph(ki,T^k)$ is the unique path between $k$ and $i$ in $T^k$ and $\left(U_e\right)_{e\in E^k}$ are independent random variables.
\citet{EV2020} show that for extremal tree models the extremal variogram 
\begin{align}
    \Gamma^{(k)}=\left( \text{Var} (W^k_i-W^k_j)\right)_{i,j\in V} \label{eq:extremal_variogram}
\end{align}
forms a tree metric, i.e.~$\Gamma^{(k)}_{ij}=\sum_{e\in\ph(ij,T^k)}\Gamma^{(k)}_{e}$.
Assuming $\Gamma^{(k)}_{ij}> 0$ for all $i\neq j$, this yields that the minimum spanning tree given weights $\Gamma^{(k)}_{ij}$ equals $T$.
This means that for extremal tree models the extremal variogram encodes their conditional independences as it directly corresponds to the underlying graphical structure. 
\cite{EV2020} apply this for nonparametric consistent graph structure learning for extremal tree models (under mild assumptions) by consistently estimating the extremal variogram.
\begin{ex}\label{ex:claw}
    For the (uncolored) claw graph on four nodes as in Figure~\ref{fig:claw}, let $k=1$. Then, we compute that the first extremal function is
    $\W^1=(0,U_{14}+U_{42},U_{14}+U_{43},U_{14})^T$
        and the extremal variogram is
\[\Gamma^{(1)}=\begin{pmatrix}
            0&\Gamma_{14}^{(1)}+\Gamma_{24}^{(1)}&\Gamma_{14}^{(1)}+\Gamma_{34}^{(1)}&\Gamma_{14}^{(1)}\\
            \Gamma_{14}^{(1)}+\Gamma_{24}^{(1)}&0&\Gamma_{24}^{(1)}+\Gamma_{34}^{(1)}&\Gamma_{24}^{(1)}\\
            \Gamma_{14}^{(1)}+\Gamma_{34}^{(1)}&\Gamma_{24}^{(1)}+\Gamma_{34}^{(1)}&0&\Gamma_{34}^{(1)}\\
            \Gamma_{14}^{(1)}&\Gamma_{24}^{(1)}&\Gamma_{34}^{(1)}&0\\
        \end{pmatrix}.
    \]
\end{ex}

\subsection{Signed Laplacian-constrained Gaussian graphical models}\label{sec:expfam}
 Let $ G=(V,E) $ be a simple undirected graph with vertex set $ V=\{1,\ldots,d\} $ and edge set $ E\subset V\times V $. We assign edge weights $ Q_{ij} $ for all edges $ ij\in E $ and assume a zero weight for non-edges.
Let $ D $ be a diagonal matrix with $ D_{ii}=\sum_{j\neq i}Q_{ij} $. The signed Laplacian matrix of $ G $ is the difference $\Theta := D -Q$.
When $Q_{ij}>0$ for all $ij\in E$, the matrix $\Theta$ is called a Laplacian matrix.
The terminology \textit{signed} Laplacian matrix stems from the more common restriction to graphs with positive edge weights where we speak of Laplacian matrices. In this work, we will allow for both positive and negative edge weights.
Note that $D$ is referred to as diagonal degree matrix and $Q$ as weighted adjacency matrix of the graph.
Clearly the row and column sums of $\Theta$ must vanish, i.e.~each row or column vector lives in the hyperplane $\mathcal{H}^{d-1}:=\{\x\in\RR^d:\x^T\mathbf{1}=0\}$ that is perpendicular to the all ones vector $\mathbf{1}$.
The (signed) Laplacian matrix is a central object of study in the field of graph theory as it encodes structural data about the underlying graph \citep{brouwer2011}. In particular, the spectrum of a weighted Laplacian matrix can be used to uncover the structure of biological and social networks and complex dynamical systems \citep{bolla2011, poignard2018}.

In this work we investigate the application of Gaussian models with signed Laplacian concentration matrices in graphical extremes.
Such a degenerate Gaussian random vector $\W$ is called a (signed) Laplacian-constrained Gaussian graphical model, which are a very active research topic in the graph signal processing literature, see e.g.~\citet{EHPO2017,YdMCP2020}.
The random vector $\W$ is supported on the hyperplane $\mathcal{H}^{d-1}$ and allows a density
\begin{align*}
	f(\w)&= (2\pi)^{-d/2} \Det(\Theta)^{\frac{1}{2}}\exp\left(-\frac{1}{2}(\w-\mu)^T \Theta (\w-\mu)\right)
\end{align*} 
for $ \w \in \mathcal{H}^{d-1} $, where the concentration matrix $\Theta$ is a positive semidefinite signed Laplacian matrix and $\mu \in \mathcal{H}^{d-1}$ is the mean vector. Here we use $\Det(\Theta)$ to denote the \emph{pseudo-determinant} of $\Theta$, which is the product of its nonzero eigenvalues.
Note that when all edge weights are positive such that $\Theta$ is a Laplacian matrix, we call $\W$ a Laplacian-constrained Gaussian graphical model.
Assuming a fixed mean vector $\mu$, the degenerate Gaussian $\W$ is an exponential family 
with natural parameter $ Q $ and mean parameter $ -\Gamma/2 $ where
$\Gamma_{ij}=\EE\left[(W_i-W_j)^2\right],$
compare \citet{RS2022}.
The matrix $\Gamma$ is a variogram or Euclidean distance matrix, i.e.~its entries correspond to squared Euclidean distances between $d$ points in $\RR^{d-1}$.
The mean parameter $-\Gamma/2$ and the natural parameter $Q$ are connected via a one-to-one transformation, see also \citet{HES2022,devriendt2020effective}.
Let $\Sigma=\Theta^+$ be the pseudo-inverse of the signed Laplacian matrix $\Theta$ of the graph $G$, i.e.~$\Sigma$ is the covariance of $\W$. The variogram matrix is recovered via $\Gamma=d_{\Sigma}\bs 1^T +\bs 1 d_\Sigma^T - 2\Sigma,$
where $d_\Sigma$ is the diagonal of $\Sigma$ in vector notation.
To compute the covariance $\Sigma$ from the variogram $\Gamma$, we define 
$\bs P:=I_d-\tfrac1d\bs 1\bs 1^T$
to be the projection matrix that maps points from $\RR^d$ to $\mathcal{H}^{d-1}$ along the all-ones vector $\mathbf{1}$.
It follows that  $\Sigma = P\left(-\frac{\Gamma}{2}\right)P$.
Let $\mathbb{S}_0^d$ define the space of symmetric $d\times d$-matrices with zero diagonal.
The natural parameter space of the exponential family $f(\w)$ is given by $\mathcal{Q}=\{Q\in \mathbb{S}_0^d:\;\int_{\mathcal{H}^{d-1}}f(\w)d\w<\infty \}$.
By Kirchhoff's matrix tree theorem the log-partition function equals 
$A(Q)=-\frac{1}{2}\log \sum_{T\in \mathcal{T}}\prod_{ij\in T} Q_{ij}, $
 where $ \mathcal{T} $ is the set of all spanning trees. 

Assume $ n $ independent observations $ \w_1,\ldots,\w_n $ of $\W$. 
We obtain a sample variogram $\widetilde{\Gamma}_{ij}= \frac{1}{n}\sum_{k=1}^{n} (w_{ki}-w_{kj})^2 $
as the summary statistic of the exponential family $f(\w)$.
This gives rise to the log-likelihood
\begin{align}
		\ell(Q;\widetilde{\Gamma})&:=-\frac{1}{2} \langle \widetilde{\Gamma}, Q \rangle-A(Q)=\frac{1}{2}\log \sum_{T\in \mathcal{T}}\prod_{ij\in T} Q_{ij} - \frac{1}{2}\langle \widetilde{\Gamma},Q \rangle
\end{align}	
with $ \langle \Gamma,Q \rangle:= \sum_{i<j} \Gamma_{ij} Q_{ij} $.

Maximum likelihood estimation for exponential families is equivalent to minimizing the Kullback-Leibler divergence \citep[Remark~6.4]{Brown1986}. 
To compute the Kullback--Leibler divergence for the given exponential family, we require the Fenchel conjugate of the log-partition function $ A(Q) $. This equals 
$
A^*(\Gamma)= \sup_{Q\in \mathbb{S}_0^d}\ell(Q;\Gamma)=-\frac{d-1}{2}-\frac{1}{2}\log\det\CM(\Gamma),
$
where
\begin{align}
    \CM(\Gamma):=\left(\begin{array}{ll}
       -\frac{\Gamma}{2}&\bs 1\\
       -\bs 1^T&0\\
        \end{array}\right).\label{eq:CM}
\end{align}
is the Cayley--Menger matrix \citep{devriendt2020effective}.

Assume two signed Laplacian-constrained Gaussian graphical models with mean parameters $\Gamma_1,\Gamma_2$ and natural parameters $Q_1,Q_2$.
This gives rise to a closed formula for the Kullback--Leibler (KL) divergence between these two distributions \citep[Proposition~6.3]{Brown1986} that is given by
\begin{align}\label{kldivergence}
	\textrm{KL}(\Gamma_1, Q_2) &= \frac{1}{2} \langle \Gamma_1, Q_2 \rangle + A^*(\Gamma_1) + A(Q_2).
\end{align}
Maximizing the log-likelihood with respect to sample variogram $\tilde{\Gamma}$ is equivalent to minimizing $\textrm{KL}(\widetilde{\Gamma}, Q)$ in $Q$.
The reciprocal operation, i.e.~minimizing the KL-divergence $\textrm{KL}(\Gamma, \widetilde{Q})$ in $\Gamma$, yields a different convex optimization problem that is known as dual or reciprocal maximum likelihood estimation. 
The dual log-likelihood function is the negative of the terms in \eqref{kldivergence} that depend on $\Gamma_1$, such that
\begin{align}
 \widecheck{\ell}(\Gamma; \widetilde{Q})&:=A^*(\Gamma)-\langle {\Gamma}, \widetilde{Q} \rangle\;\propto\; \frac{1}{2}\log\det\CM(\Gamma)-\frac{1}{2}\langle {\Gamma}, \widetilde{Q} \rangle.\label{eq:dual_loglik}    
\end{align}	
The dual log-likelihood allows for simple handling of convex constraints in the mean parameter $\Gamma$, while for the log-likelihood this is true for convex constraints in the natural parameter $Q$. This structure is exploited in mixed dual estimation as introduced by \cite{LZ2022}, which we will briefly explain below for our setting.

Assume a graph $G=(V,E)$.
According to \citet{BN1978}, the mixed parameterization $\left((\Gamma_{ij})_{ij\in E}, (Q_{ij})_{ij\in E^c}\right)$ is a valid parameterization of the exponential family, compare also \cite{LZ2022} for details on mixed parameterizations.
The mixed parameterization allows for a mixed dual estimator (MDE) that solves two consecutive convex problems that minimize the Kullback--Leibler divergence in two steps.
Let $ C_E\subseteq M $ denote convex constraints on $ (\Gamma_{ij})_{ij\in E} $ and $ C_{E^c}\subseteq \mathcal{Q} $ convex constraints on $ (Q_{ij})_{ij\in E^c} $.
The MDE is a two-step estimation procedure such that
	\begin{enumerate}
		\item the first step~(S1) is defined as
		\[\widehat{Q}=\argmax_{Q \in C_{E^c}} \ell(Q;\widetilde{\Gamma}), \] 
		\item and the second step~(S2) as
		\[\widecheck{\Gamma}=\argmax_{\Gamma\in C_E}\widecheck{\ell}(\Gamma; \widehat{Q}).\]
	\end{enumerate}
The unique optimum of step~(S1) is the maximum likelihood estimator (MLE) for the convex exponential family given by $Q \in C_{E^c}$.
Note that when $C_{E^c}$ yields only affine constraints the MDE constitutes a maximum dual likelihood estimator \citep[Remark~4.7]{LZ2022} as studied for example by \cite{Chr1989}.
In the context of signed Laplacian-constrained Gaussian graphical models,
\citet{RS2022} apply the MDE under affine constraints (in both steps) that encode an extremal notion of positive dependence for H\"usler--Reiss distributions, which are a popular family of multivariate Pareto distributions that we will discuss in the following section.

\subsection{H\"usler--Reiss distributions}\label{sec:prelim_HR}
The H\"usler--Reiss distribution \citep{HR1989} is considered as the Gaussian analogue in multivariate extremes, for example due to similar properties with respect to conditional independence, compare e.g.~\citet{HES2022}.
The H\"usler--Reiss distribution is parameterized by a conditionally negative definite variogram matrix $\Gamma$, i.e.~an element of the parameter set
\[\mathcal{D}^d:=\{\Gamma\in\mathbb{S}^d_0: \x^T\Gamma \x<0 \text{ for all } \x \in \mathcal{H}^{d-1}\setminus \mathbf{0} \}. \]
An elegant construction for the H\"usler--Reiss distribution is via their extremal functions.
For some $\Gamma \in \mathcal{D}^d$ let $\Sigma^{(k)}:=\frac{1}{2}(\Gamma_{ik}+\Gamma_{jk}-\Gamma_{ij})_{i,j\neq k} $. Let $\W^k_{\setminus k}$ be a multivariate Gaussian with mean $ -\text{diag}(\Sigma^{(k)})/2 $ and covariance $\Sigma^{(k)}$. 
Then $\Y$ is H\"usler--Reiss with parameter matrix $\Gamma$, where
\begin{align}
	\begin{cases}
		\Gamma_{ik}=\Sigma^{(k)}_{ii}, &i\neq k,\\
		\Gamma_{ij}=\Sigma^{(k)}_{ii}+\Sigma^{(k)}_{jj}-2\Sigma^{(k)}_{ij}, &i, j\neq k.\\
	\end{cases}\label{eq:sigma2gamma}
\end{align}
Note that by construction, the extremal variogram \eqref{eq:extremal_variogram} is the parameter matrix $\Gamma$ for all $1\le k\le d$. 
\citet{HES2022} introduced a characterization of the H\"usler--Reiss distribution without choosing a particular $k\in [d]$ via the relation
$\Y|\{\Y^T\mathbf{1}>0\}\stackrel{d}{=} E\mathbf{1}+\W.$
If $\Y$ is parameterized by $\Gamma$, then $\W$ is a signed Laplacian-constrained Gaussian graphical model with mean $P(-\Gamma/2)\mathbf{1}$ and concentration matrix $\Theta=D-Q$. 
Note that $\W$ is independent of the standard exponential random variable $E$.
It was shown by \cite{HES2022} that the H\"usler--Reiss concentration matrix $\Theta$ directly encodes extremal conditional independence where we condition on all remaining variables, such that
\begin{align}
    \Y_i\perp_e \Y_j\vert \Y_{\setminus ij} \Longleftrightarrow \Theta_{ij}=0\Longleftrightarrow Q_{ij}=0.\label{eq:HR-precisionCI}
\end{align}
Therefore the extremal graphical structure of a H\"usler--Reiss random vector $\Y$ is encoded in the concentration matrix and the weighted adjacency matrix of the corresponding signed Laplacian-constrained Gaussian graphical model~$\W$.
This enables parametric surrogate likelihood inference for H\"usler--Reiss graphical models, see e.g.~\cite{RS2022} and Section~\ref{sec:surrogate_lik}.

\section{Extremal conditional independence for H\"usler--Reiss distributions}\label{sec:extrCI}

Parametric descriptions of arbitrary extremal conditional independence statements for H\"usler--Reiss distributions are only known for statements via the Gaussian $k$-th extremal function $\W^k_{\setminus k}$.
Furthermore, statements of the form $\Y_i\perp_e \Y_j\vert \Y_{\setminus ij}$ correspond to vanishing entries $\Theta_{ij}=0$ in the H\"usler--Reiss precision matrix.
In this section we introduce a parametric characterization of arbitrary extremal conditional independence statements for H\"usler--Reiss distributions via their parameter matrices.
This is based on the Cayley--Menger matrix 
$ \CM(\Gamma) = \left(\begin{smallmatrix}       
    -\frac{\Gamma}{2}&\bs 1\\
    -\bs 1^T&0\\
\end{smallmatrix}\right), $ 
see \eqref{eq:CM}.
The determinant of this matrix is given by $\det(\CM(\Gamma))=\det(\Sigma^{(k)})$, which further relates to the volume of the simplex spanned by the Euclidean distances stored in $\Gamma$, compare \citet{devriendt2020effective}. 
A key observation for a parametric description of extremal conditional independence for H\"usler--Reiss distributions is the following rephrasing of Fiedler's identity:

\begin{lemma}\label{lem:CMinv}
For a variogram matrix $\Gamma$ in $\mathcal{D}^d$ we can obtain the inverse of its Cayley-Menger matrix by
\begin{equation} 
\label{eq:gammaaugmented}
\CM(\Gamma)^{-1}=\begin{pmatrix}
	\Theta&-\bs r \\
     \bs r^T& -R^2 \\
\end{pmatrix},
\end{equation}
where $\bs r=\tfrac{1}{2}\Theta \xi+\tfrac{1}{d}\bs 1$ with $\xi={\rm diag}(\Sigma)$, and $R=\sqrt{\tfrac{1}{2}\xi^T (\bs r+\tfrac{1}{d}\bs 1)}$.
\end{lemma}
\begin{proof}
By Fiedlers identity \citep[Theorem~2]{devriendt2020effective} we have that
\begin{equation}\label{eq:fiedlerid}
-\frac{1}{2}\begin{pmatrix}
	\Gamma&\bs 1\\
	\bs 1^T &0 \\
\end{pmatrix}\;=\;\begin{pmatrix}
	\Theta&-2\bs r \\
     -2\bs r^T& 4R^2 \\
\end{pmatrix}^{-1},
\end{equation}
where $\bs r=\tfrac{1}{2}\Theta \xi+\tfrac{1}{d}\bs 1$ with $\xi={\rm diag}(\Sigma)$, and $R=\sqrt{\tfrac{1}{2}\xi^T (\bs r+\tfrac{1}{d}\bs 1)}$. Note that \cite{devriendt2020effective} only refers to Laplacian matrices in the theorem, but extends to signed Laplacians (or arbitrary Euclidean distance matrices, respectively) in the subsequent remark.
Because of 
\[\CM(\Gamma)=
\diag(1,\dots,1,2)
\begin{pmatrix}
	-\frac{\Gamma}{2}&-\frac{1}{2}\bs 1\\
	-\frac{1}{2}\bs 1^T &0 \\
\end{pmatrix}
\diag(1,\dots,1,-2)
\]
we obtain
\begin{align*}
    \CM(\Gamma)^{-1}&=
\diag(1,\dots,1,-\frac{1}{2})
\begin{pmatrix}
	\Theta&-2\bs r \\
     -2\bs r^T& 4R^2 \\
\end{pmatrix}
\diag(1,\dots,1,\frac{1}{2}) \\
&=\begin{pmatrix}
	\Theta&-\bs r \\
     \bs r^T& -R^2 \\
\end{pmatrix}.
\end{align*}
\end{proof}
From Fiedler's identity, we derive the following relation for subdeterminants of the covariance of $\W^k$ and subdeterminants of the Cayley--Menger matrix:
\begin{lemma}
For disjoint $i,j,C$ and for $i,j\neq k$ and $k\not\in C$ we have that
\begin{align}
    \det(\Sigma^{(k)}_{iC,jC})&= \det\left(\CM(\Gamma)_{ikC(d+1),jkC(d+1)}\right).
\end{align}
\end{lemma}
\begin{proof}
The inverse Farris transform of $\Sigma^{(k)}_{ijC}$ is $\Gamma_{ijkC}$.
Note that $\Theta^{(k)}_{ij}=\Theta_{ij}$ for $i,j\neq k$ by definition.
Therefore the $(i,j)$-th entry of the inverse of $\Sigma^{(k)}_{ijC}$ equals the $(i,j)$-th entry of the inverse of $\CM(\Gamma)_{ijkC(d+1)}$ by Lemma~\ref{lem:CMinv}.
With Cramer's rule, we obtain that
\begin{align*}
    \frac{(-1)^{i+j}\det(\Sigma^{(k)}_{iC,jC})}{\det(\Sigma^{(k)}_{ijC}))}&=\frac{(-1)^{i+j}\det(\CM(\Gamma)_{ikC,jkC})}{\det(\CM(\Gamma)_{ijkC}))}.
\end{align*}
By the Cayley--Menger determinant formula, the proposition follows.
\end{proof}
The previous result yields a general parametric characterization of extremal conditional independence for H\"usler--Reiss distributions:
\begin{prop}\label{prop:CIviaGamma}
For a H\"usler--Reiss random vector $\Y$ with parameter matrix $\Gamma$ the following equivalence holds:
\[Y_i\perp_e Y_j\vert \Y_C \Longleftrightarrow \det(\CM(\Gamma)_{iC(d+1),jC(d+1)})=0. \]
\end{prop}

We further find a characterization of $\Gamma$ in the parameters of $Q$ that will be particularly useful for tree models. For a given graph $G$ and distinct vertices $i$ and $j$ in $G$, let $\mathcal{F}_{ij}$ denote the set of all spanning forests $F$ of $G$ such that (1) $F$ has exactly two trees and (2) $i$ and $j$ are in different trees of $F$. Let $\mathcal{T}$ denote the set of all spanning trees of $T$.

\begin{lemma}\label{lem:GammainQ}
Let $\Gamma$ be a Euclidean distance matrix and $Q$ the weighted adjacency matrix of the corresponding signed Laplacian $\Theta=D-Q$.
\begin{enumerate}
    \item The $ij$ entry of $\Gamma$ is
\[\Gamma_{ij}=\frac{\sum_{F\in\mathcal{F}_{ij}}\prod_{st\in F}Q_{st}}{\sum_{T\in\mathcal{T}}\prod_{st\in T}Q_{st}}. \]

\item When $Q$ is the weighted adjacency matrix of a tree $T$, it follows $\Gamma_{ij}=\frac{1}{Q_{ij}}$ for all $ij\in T$.
\end{enumerate}
\end{lemma}

\begin{proof}
    (1) Let $i$ and $j$ be distinct vertices of $G$.
    We have that $\Gamma_{ij} = \Sigma_{ii}^{(j)}$. So we wish to compute the principle minor of $\Theta$ obtained by removing the $i$th and $j$th rows and columns. By the All-Minors Matrix Tree Theorem \citep{chaiken1982}, we have that
    \[
    \det(\Theta_{/ij,/ij}) = \sum_{F \in \mathcal{F}_{ij}} \prod_{st \in F} Q_{st}.
    \]
    We note that while the full version of the All-Minors Matrix Tree Theorem involves a specification of signs for each forest, the minor that we wish to compute is principle, and hence each of these signs are 1. By applying the standard Matrix Tree Theorem for weighted graphs, we see that
    \[
    \det(\Theta_{/j,/j}) = \sum_{T \in \mathcal{T}} \prod_{st \in T} Q_{st},
    \]
    and the desired formula for $\Gamma_{ij}$ follows immediately.
    
    (2) Let $T$ be a tree and let $i$ and $j$ be adjacent vertices in $T$. Then $T$ has exactly one spanning tree, namely itself, and $ij$ is an edge in it. Moreover $\mathcal{F}_{ij}$ contains exactly one forest; this forest is exactly $T$ with the edge $ij$ removed, and in particular, it contains all other edges of $T$. So the desired expression for $\Gamma_{ij}$ follows directly from part (1).
\end{proof}

\section{Colored extremal graphical models}
In this section, we will introduce colored extremal graphical models.
We first collect some necessary notation.
Let $ G $ be a simple undirected graph with vertex set $ V=[d] $ and edge set $ E\subset V\times V $.
Define a function $ \lambda: E\rightarrow [r] $ that maps every edge to a natural number $1,\ldots,r$ where $r\le |E|$.
We call $ \mathcal{G} $ a colored graph based on $ G $ with $ r $ edge color classes and 
graph coloring function $ \lambda $, such that $ \lambda(e)=\lambda(e') $ for edges $ e,e'\in E $ with the same color. 
Let $E_i$ denote the collection of nodes in the $i$-th color class, such that $E_i\cap E_j=\emptyset $ for all $i\neq j$ and $E_1\cup\ldots\cup E_r=E$.

\subsection{Colored extremal tree models}\label{sec:def_col_tree}

The particularly nice structure of extremal tree models allows for nonparametric colored tree models.
Assume a tree $T=(V,E)$ with coloring function $\lambda$.
Recall from Section~\ref{sec:ext_tree} that a multivariate Pareto vector $\Y$ that is Markov to $T$ can be characterized by a vector of independent univariate random variables $(U_e)_{e\in E^k}$, where $T^k=(V,E^k)$ is a directed tree obtained from $T$ by directing all edges away from $k$.

We define a nonparametric colored extremal tree model by imposing the same distribution for all univariate random variables $U_e$ within the same edge color class. Therefore we call $\Y$ a nonparametric colored extremal tree model with respect to $(T,\lambda)$ when
\begin{enumerate}
    \item $\Y$ is extremal Markov to $T$,
    \item $U_e\stackrel{d}{=}U_{e'}$, whenever $ \lambda(e)=\lambda(e') $ for $ e,e'\in E^k $.
\end{enumerate}
As a consequence, we find that the entries in the extremal variogram~\eqref{eq:extremal_variogram} for edges within the same edge color class are identical, i.e.
\begin{align*}
    \Gamma^{(k)}_{st}=\text{Var}(U_{st})=\text{Var}(U_{uv})=\Gamma^{(k)}_{uv}\qquad \text{ when } \lambda(st)=\lambda(uv) \text{ for } st,uv\in E^k.
\end{align*}
This means that the extremal variogram allows to identify edge colorings.
\begin{ex}
    Let $\Y$ be a colored extremal tree model with respect to the colored claw graph in Figure~\ref{fig:claw} for $k=4$, see also Example~\ref{ex:claw}.
    This means that $U_{41}\stackrel{d}{=}U_{43}$ and $\Gamma^{(4)}_{14}=\Gamma^{(4)}_{34}$.
\end{ex}
Bivariate Pareto distributions can impose an extremal tree model, see \cite{EV2020} for details. We employ this fact to construct a colored extremal tree model in the following example.
\begin{ex}
Assume a multivariate Pareto distribution $\Y$ that is Markov to some tree $T=(V,E)$ where every bivariate marginal $\Y_e,\;e\in E$ is distributed according to an extremal logistic distribution.
This implies for $e\in E$ that $U_e$ can be stochastically represented by a difference of a Fréchet-distributed random variable and a Gamma-distributed random variable depending on a parameter $\theta_e\in(0,1)$, see \citet[Example~2]{EV2020}.
Then it holds that $\Gamma^{(k)}_{uv}=\theta_{uv}^2(\psi^{(1)}(1-\theta_{uv})+\pi^2/6) $ for $uv\in E$, where $\psi^{(1)}$ is the trigamma function.
We observe that symmetries $U_e\stackrel{d}{=}U_{e'}$ can therefore be parametrically encoded via $\theta_e=\theta_{e'}$.
\end{ex}

Colored extremal tree models provide a nonparametric complexity reduction in comparison to uncolored extremal tree models, as
that they allow for a nonparametric extremal graphical model that can be characterized by a small number of univariate random variables.
In the following sections, we will discuss different types of parametric colored graphical models for the family of H\"usler--Reiss distributions. We will see in Section~\ref{sec:col_HRtree} that for H\"usler--Reiss tree models all models coincide with the colored tree model as defined above.

\subsection{H\"usler--Reiss RCON model}\label{sec:def_RCON}

Graph colorings allow for statistical modeling of parameter symmetries, and hence parameter reduction, of graphical models.
Herein edges within the same edge color class are considered to impose equality among their respective parameters.
For colored Gaussian graphical models, an RCON model defines parameter equalities in their concentration matrices \citep{LauritzenColoured}; the name ``RCON" models reflects that such a model is constructed by imposing restrictions on the concentration matrix.
In this section we discuss a similar approach for H\"usler--Reiss graphical models.

As discussed in Section~\ref{sec:prelim_HR}, a H\"usler--Reiss distribution can be parameterized by either a variogram matrix $\Gamma$ or by a weighted adjacency matrix $Q$, where $Q$ encodes the extremal graphical structure.
We define an $ \text{RCON} $ model for H\"usler--Reiss graphical models with weighted adjacency matrix $ Q $ such that
	\begin{enumerate}
		\item $Q_{uv}=0$ for $ ij\not\in E $,
		\item $ Q_{st}=Q_{uv} $ for $ \lambda(st)=\lambda(uv) $ for $ st,uv\in E $.
	\end{enumerate}
The $ \text{RCON} $ model is a linear submodel of the regular H\"usler--Reiss graphical model. We will develop statistical methodology for H\"usler--Reiss RCON models in Section~\ref{sec:learn_RCON}.
\begin{ex}
\begin{figure}
    \centering
   	\begin{tikzpicture}[line width=.6mm]
			\node[minimum size=8mm,shape=circle,draw=black] (1) at (-2,1) {1};
			\node[minimum size=8mm,shape=circle,draw=black] (2) at (-2,-1) {2};
			\node[minimum size=8mm,shape=circle,draw=black] (3) at (0,0) {3};
			\node[minimum size=8mm,shape=circle,draw=black] (4) at (2,1) {4};
   \node[minimum size=8mm,shape=circle,draw=black] (5) at (2,-1) {5};
			\path[color=red] (1) edge (2);
			\path[color=blue] (1) edge (3);
			\path[color=blue] (2) edge (3);
			\path[color=green] (3) edge (4);			
			\path[color=green] (3) edge (5);
   			\path[color = red] (4) edge (5);
		\end{tikzpicture}
    \caption{Colored graph with five nodes and six edges}
    \label{fig:example}
\end{figure}
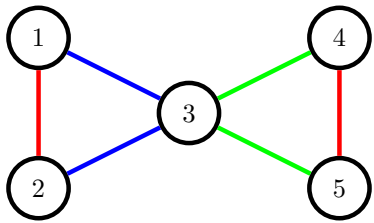
A H\"usler--Reiss RCON model with respect to the colored graph in Figure~\ref{fig:example} has a weighted adjacency matrix
\begin{align*}
    Q=\begin{pmatrix}
        0&\textcolor{red}{\omega_1}&\textcolor{blue}{\omega_2}&0&0\\
        \textcolor{red}{\omega_1}&0&\textcolor{blue}{\omega_2}&0&0\\
        \textcolor{blue}{\omega_2}&\textcolor{blue}{\omega_2}&0&\textcolor{green}{\omega_3}&\textcolor{green}{\omega_3}\\
        0&0&\textcolor{green}{\omega_3}&0&\textcolor{red}{\omega_1}\\
        0&0&\textcolor{green}{\omega_3}&\textcolor{red}{\omega_1}&0\\
    \end{pmatrix}.
\end{align*}
It is parameterized by an ordered triple $\omega=(\omega_1,\omega_2,\omega_3)$.
\end{ex}

\subsection{Variogram symmetries in arbitrary H\"usler--Reiss graphical models}\label{sec:def_RVAR}
In the discussion of \cite{EH2020} by Huser and Cisneros, they discuss parameter symmetries for H\"usler--Reiss tree models, where all entries of $\Gamma$ that correspond to edges of the tree are identical.
This approach is proposed by Huser and Cisneros as a new method to impose spatial structure for data on regular lattices.
We generalize this approach to allow for parameter symmetries in $\Gamma$ for arbitrary colored graphs.
Therefore assume a graph $G=(V,E)$ with edge coloring function $\lambda$.
We define variogram symmetries via $\lambda$ for a H\"usler--Reiss graphical model, such that
\begin{enumerate}
    \item $\Gamma_{st}=\Gamma_{uv}$  when $\lambda(st)=\lambda(uv)$ for $st,uv\in E$,
    \item $Q_{uv}=0$ when $uv\not\in E$.
\end{enumerate}
We refer to this model as a H\"usler--Reiss RVAR model, as it imposes parameter symmetries in the variogram matrix.
While this model leads to highly nonlinear constraints when parameterized either in $\Gamma$ or in $Q$, a mixed parameterization $\left((\Gamma_{uv})_{uv\in E},(Q_{uv})_{uv\not\in E}\right)$ allows for a linear description.
Statistical methodology for this model based on mixed parameterizations will be discussed in Section~\ref{sec:learn_RVAR}.

\begin{ex}

A H\"usler--Reiss RVAR model with respect to the colored graph in Figure~\ref{fig:example} has a mixed parameterization that is characterized by a variogram matrix and precision matrix
\begin{align*}
    \Gamma&=\begin{pmatrix}
        0&\textcolor{red}{\nu_1}&\textcolor{blue}{\nu_2}&*&*\\
        \textcolor{red}{\nu_1}&0&\textcolor{blue}{\nu_2}&*&*\\
        \textcolor{blue}{\nu_2}&\textcolor{blue}{\nu_2}&0&\textcolor{green}{\nu_3}&\textcolor{green}{\nu_3}\\
        *&*&\textcolor{green}{\nu_3}&0&\textcolor{red}{\nu_1}\\
        *&*&\textcolor{green}{\nu_3}&\textcolor{red}{\nu_1}&0\\
    \end{pmatrix},\qquad Q=\begin{pmatrix}
        0&*&*&0&0\\
        *&0&*&0&0\\
        *&*&0&*&*\\
        0&0&*&0&*\\
        0&0&*&*&0\\
    \end{pmatrix}.
\end{align*}
Therefore the given RVAR model is parametrized by an ordered triple $\nu=(\nu_1,\nu_2,\nu_3)$.
The missing entries in $\Gamma$ and $Q$ can be determined uniquely by solving the matrix completion problem \citep[Proposition~4.3]{HES2022}.
Note that because $G$ is a decomposable graph in this example, the matrix completion problem is solved by rational functions in $\nu$.
\end{ex}

\subsection{H\"usler--Reiss colored tree model}\label{sec:col_HRtree}
Let $\Y$ be a H\"usler--Reiss tree model with parameter matrix $\Gamma$ with respect to some tree $T=(V,E)$.
The $k$th extremal function $\W^k$ allows a construction \eqref{eq:tree_construction} from univariate Gaussian variables $U_{ij}\sim N(-\Gamma_{ij}/2,\Gamma_{ij}) $ for all $ij\in E^k$.
We define a H\"usler--Reiss colored tree model for some coloring function $\lambda$ by imposing the relation $U_e\stackrel{d}{=}U_{e'}$, whenever $ \lambda(e)=\lambda(e') $ for $ e,e'\in E^k $, compare Section~\ref{sec:def_col_tree}.
As the relations $U_e\stackrel{d}{=}U_{e'}$ are equivalent to imposing variogram symmetries, we observe an equivalence to H\"usler--Reiss RVAR tree models.
By Lemma~\ref{lem:GammainQ} we find that the definitions for H\"usler--Reiss RCON and RVAR models on trees coincide, which yields the following proposition.
\begin{prop}\label{prop:colHRtree} Let  $T = (V,E)$ be a tree. 
    The colored H\"usler--Reiss tree models, H\"usler--Reiss RCON models and  H\"usler--Reiss RVAR models that are Markov to $T$ coincide.
\end{prop}

\begin{proof}
    The equivalence between colored H\"usler--Reiss tree models on $T$ and H\"usler--Reiss RVAR models on $T$ follows directly from the definitions.
    Let $\Y$ be a H\"usler--Reiss RCON vector with respect to a tree $T=(V,E)$.
    Hence, $\Y$ allows a construction by univariate Gaussians $(U_{ij})_{ij\in E^k}$ with variances $\Gamma_{ij}$, where the remaining entries of $\Gamma$ can be recovered via the tree metric, compare Section~\ref{sec:ext_tree}. Since $\Gamma$ is the Euclidean distance matrix of a non-degenerate simplex, its entries are nonzero, compare \cite{devriendt2020effective}.
    According to Lemma~\ref{lem:GammainQ} it then follows that 
    \[Q_{ij}=\frac{1}{\Gamma_{ij}}=\frac{1}{\Gamma_{st}}=Q_{st}\]
    when $\lambda(ij)=\lambda(st)$ for $ij,st\in E$.
    This proves the proposition.
\end{proof}

\section{Learning colored H\"usler--Reiss graphical models}\label{sec:surrogate_lik}

In this section we introduce methodology for surrogate likelihood inference for RCON and RVAR models.
Hereby we study likelihood inference for signed Laplacian-constrained Gaussian graphical models under RCON and RVAR constraints and then discuss their application in multivariate extremes for H\"usler--Reiss distributions.

As illustrated in Section~\ref{sec:prelim_HR}, the extremal functions of a H\"usler--Reiss distribution are multivariate Gaussians where both the mean and the covariance depend on the parameter matrix $\Gamma$.
The parameter dependence in the mean leads to analytical complications when interest is in maximum likelihood estimation for H\"usler--Reiss distributions.
This motivates surrogate likelihood inference where we maximize the log-likelihood of a signed Laplacian-constrained Gaussian graphical model (compare e.g.~\citet{REZ2021,HES2022}) instead of the actual H\"usler--Reiss log-likelihood.
Hereby we ignore the likelihood contribution of the mean, such that the resulting likelihood problem allows for a simple scoring algorithm.
Equivalently, a surrogate Kullback--Leibler divergence for signed Laplacian-constrained Gaussian graphcial models gives rise to a surrogate mixed dual estimator.
This yields a two-step estimation procedure when convex constraints are applied to a mixed parameterization (compare Section~\ref{sec:expfam}), which we will employ for H\"usler--Reiss RVAR models.

\subsection{Signed Laplacian-constrained Gaussian graphical models under RCON constraints}\label{sec:learn_RCON}
As discussed in Section~\ref{sec:expfam}, a signed Laplacian-constrained Gaussian graphical model forms an exponential family with mean parameter $-\Gamma/2$ and natural parameter $Q$.
Given a graph $G=(V,E)$, assume an RCON model with coloring function $\lambda:E\to [r]$ as in Section~\ref{sec:def_RCON}
and define a colored signed Laplacian-constrained Gaussian graphical model through constraints	 
\[Q_{ij}=0 \text{ for }  ij\not\in E , \quad Q_{ij}=Q_{uv}  \text{ for } \lambda(ij)=\lambda(uv)  \text{ for }  ij,uv\in E. \]
The model obtained by imposing linear constraints on the natural parameters of an exponential family is itself an exponential family.
Therefore the colored signed Laplacian-constrained Gaussian graphical model forms an exponential family with natural parameter $\omega=(\omega_{1},\ldots,\omega_r) $, and natural parameter space $\Omega:=\{\omega\in\RR^r:\;Q(\omega)\in\mathcal{Q} \}$, where $Q(\omega)$ is given by
\begin{align}
    Q(\omega)_{uv}=\begin{cases}
        \omega_i & \ uv \in E: \;\lambda(uv)=i,\\
        0 & \ uv \not\in E.
    \end{cases}\label{eq:Qinomega}
\end{align}
The sufficient statistic of a signed Laplacian-constrained Gaussian graphical model with respect to $G$ is $(-\widetilde{\Gamma}_{uv}/2)_{uv \in E}$, see Section~\ref{sec:prelim_HR}. 
The linear constraints of the RCON model as in \eqref{eq:Qinomega} yield a sufficient statistic
$\mathbf{t}(\widetilde{\Gamma})$ where
\begin{align}
     t_i(\widetilde{\Gamma})=-\frac{1}{2}\sum_{uv \in E: \;\lambda(uv)=i} \widetilde{\Gamma}_{uv}.\label{eq:summary_stat}
\end{align}
This yields the log-likelihood
\begin{align*}\ell(\omega;\mathbf{t}(\widetilde{\Gamma}))=\frac{1}{2}\log \sum_{T\in \mathcal{T}}\prod_{ij\in T} Q_{ij}(\omega) +\langle \mathbf{t}(\widetilde{\Gamma}),\omega \rangle.   
\end{align*}
We define a maximum likelihood estimator for the signed Laplacian-constrained Gaussian graphical model under RCON constraints given $\mathbf{t}(\widetilde{\Gamma})$ as
\begin{align}
    \widehat{\omega}&= \argmax_{\omega \in \Omega}\ell(\omega;\mathbf{t}(\widetilde{\Gamma})). \label{eq:mlecol}
\end{align}
The partial derivatives of the log-likelihood with respect to $\omega_i$ are the score functions
\begin{align*}
    S_i(\omega)&=\frac{\partial}{\partial \omega_i} \ell (\omega,\mathbf{t}(\widetilde{\Gamma}))=\frac{1}{2}\sum_{uv \in E: \;\lambda(uv)=i} \left(\Gamma_{uv}(\omega)- \widetilde{\Gamma}_{uv}\right), \quad 1\le i \le r.
\end{align*}
Together with graphical constraints $Q_{ij}=0 $ for all $ij\not\in E$, these characterize the maximum likelihood estimator $\widehat{\omega}$.
When interest is in finding a maximum likelihood estimate $\widehat{\omega}$ given the sufficient statistic $\mathbf{t}(\widetilde{\Gamma})$, classical approaches include scoring algorithms, which make use of Newton's method to find the vanishing point of the score equations.
Such algorithms require the Hessian matrix of the log-likelihood, which for exponential families equals the negative of the information matrix of the given model.
We therefore compute the information matrix of the RCON signed Laplacian Gaussian graphical model in the following lemma.
\begin{lemma}\label{lem:infmat}
The information matrix $I(\omega)\in \RR^{r\times r}$ of a colored signed Laplacian-constrained Gaussian graphical model is
    \begin{align*}
    I(\omega)_{ij}&=\frac{1}{8}\sum_{uv \in E: \;\lambda(uv)=i}\sum_{st \in E: \;\lambda(uv)=j} (\Gamma_{sv}(\omega)-\Gamma_{us}(\omega)-\Gamma_{tv}(\omega)+\Gamma_{ut}(\omega))^2
\end{align*}
\end{lemma}
A proof of Lemma~\ref{lem:infmat} is available in Section~\ref{prf:infmat}.

\subsection{Signed Laplacian-constrained Gaussian graphical models under variogram symmetries}\label{sec:learn_RVAR}

Assume a signed Laplacian-constrained Gaussian graphical model that satisfies RVAR constraints with respect to a graph $G=(V,E)$. Therefore we assume parameter constraints
\begin{align}
    Q_{uv}=0 \text{ when }uv\not\in E,\qquad \Gamma_{st}=\Gamma_{uv}\text{ when }\lambda(st)=\lambda(uv)\text{ for }st,uv\in E,
\end{align}
in terms of the mixed parameterization $((\Gamma_{uv})_{uv\in E},(Q_{uv})_{uv\not\in E},)$. 
The affine parametric constraints on the mixed parameterization induce a mixed linear exponential family with parameter vector $\nu=(\nu_1,\ldots,\nu_r)$, where $\Gamma_{uv}(\nu)=\nu_i$ for all $uv\in E$ with $\lambda(uv)=i$.
Let $N:=\{\nu \in\RR^r: \Gamma(\nu)\in M \}$ define the mean parameter space and for a given sample adjacency matrix $\widetilde{Q}$, let $\mathbf{s}(\widetilde{Q})$ with $s_i(\widetilde{Q})=\sum_{uv \in E: \;\lambda(uv)=i}\widetilde{Q}_{uv} $ define the corresponding summary statistic.
Then, the dual log-likelihood is
\[\widecheck{\ell}(\nu,\mathbf{s}(\widetilde{Q})):=\widecheck{\ell}(\Gamma(\nu),\widetilde{Q}). \]
A mixed dual estimator with respect to the constraints above is given by a two-step estimation procedure where
	\begin{enumerate}
		\item the first step~(S1) is
		\begin{align}
		    \widehat{Q}=\argmax_{Q\in \mathcal{Q}: Q_{ij}=0 \;\forall ij\not\in E} \ell(Q;\widetilde{\Gamma}),  \label{eq:step1}
		\end{align}
		\item and the second step~(S2) is
        \begin{align}
      		\widecheck{\nu}=\argmax_{\nu \in N}\widecheck{\ell}(\nu; \mathbf{s}(\widehat{Q})).\label{eq:step2}
        \end{align}
 \end{enumerate}
The first step~(S1) is a maximum likelihood estimation problem for the graphical model given by $G$, such that $\widehat{Q}$ is the solution of the matrix completion problem
\begin{align}
    \widehat{\Gamma}_{uv}=\widetilde{\Gamma}_{uv},\quad  uv\in E, \qquad Q_{uv}=0\quad  uv\not\in E,
\end{align}
compare also \citet{HES2022}.
The second step~(S2) constitutes a maximum dual likelihood estimator for the colored model given a summary statistic $\mathbf{s}(\widehat{Q})$, see e.g.~\citet{LZ2022} and \citet{Brown1986} for details on maximum dual likelihood estimation.

For an algorithmic solution of the maximum dual likelihood problem in step~(S2), we can employ Newton's method to define a dual scoring algorithm.
This requires the dual score equations for this model, which we obtain as
\begin{align}
    \widecheck{S}_i(\nu)&=\frac{\partial}{\partial \nu_i} \widecheck{\ell}(\nu; \mathbf{s}(\widetilde{Q}))=\frac{1}{2}\sum_{uv \in E: \;\lambda(uv)=i} \left(Q_{uv}(\nu)-\widetilde{Q}_{uv}\right), \quad 1\le i \le r.
\end{align}
For a scoring algorithm, we further require the Jacobian matrix of the score function $ \widecheck{I}(\nu):=-\nabla_\nu \widecheck{S}(\nu)^T$, i.e.~the negative Hessian matrix of the dual log-likelihood.

\begin{lemma}\label{lem:infmat_recipr}
The entries of the negative Hessian matrix of the dual log-likelihood are determined by
    \begin{align*}
        \widecheck{I}(\nu)_{ij}=-\frac{1}{4}\sum_{uv \in E: \;\lambda(uv)=i}\sum_{st \in E: \;\lambda(st)=j}(\Theta_{ut}(\nu)\Theta_{vs}(\nu)+\Theta_{us}(\nu)\Theta_{vt}(\nu)).
    \end{align*}
\end{lemma}
The proof is available in Section~\ref{prf:infmat_recipr}.

\subsection{Scoring algorithms}
Linear exponential families allow for scoring algorithms to learn a surrogate maximum likelihood estimator \eqref{eq:mlecol}, compare \citet{hojsgaard2012graphical} and \citet[Section D.1.5]{Lauritzen96} for details.
Hereby the canonical parameters are updated such that the $(m+1)$-th iteration $\omega^{m+1}$ equals
$ \omega^{m+1}= \omega^{m}+I(\omega^m)^{-1}S(\omega^m).$
According to \citet{hojsgaard2012graphical} this algorithm can suffer from various instabilities. 
As an alternative, they suggest an alternative iteration based on the dual log-likelihood that is
\begin{align}
    \omega^{m+1}=\omega^m+\left(I(\omega^m) +S(\omega^m)S(\omega^m)^T\right)^{-1}S(\omega^m).\label{eq:scoring_step}
\end{align}
The additional summand $S(\omega^m)S(\omega^m)^T$ improves the stability of the algorithm.
For linear exponential families with only one parameter, the modified algorithm was shown to be globally convergent \citep{JJL1991}.

\subsubsection{Scoring algorithm under RCON constraints}

The scoring step \eqref{eq:scoring_step} gives rise to Algorithm~\ref{alg:ScoringAlg} that computes a surrogate maximum likelihood estimate for colored H\"usler--Reiss graphical models given an empirical variogram matrix $\overline{\Gamma}$ and a start value $\omega^0$, using the formulas for the scores $S(\omega)$ and the information matrix $I(\omega)$ that we computed above.
As a straightforward starting value we average the empirical values for each edge color class, i.e.
\begin{align*}
    \omega^0_i=\frac{1}{|E_i|}\sum_{uv\in E_i}\overline{Q}_{uv} \qquad \forall i \in [r].
\end{align*}
Our implementation of Algorithm~\ref{alg:ScoringAlg} in \texttt{R} can be found in the supplementary materials. 

\begin{algorithm} \label{alg:ScoringAlg}
	$ $ \\[1ex]
	\textbf{Input:} {An empirical variogram $ \overline{\Gamma} $, a graph $ G $, an edge coloring and a start value $\omega^0$.}
	
	\noindent
	\textbf{Initialize:} Compute the initial scores $S(\omega^0)$.
	
	\noindent
	\textbf{Computation:} Iterate \eqref{eq:scoring_step} until the absolute sum of the scores is below a given threshold.
	
	\noindent
	\textbf{Output:} A parameter value $\widehat{\omega}$.
\end{algorithm}
As Algorithm~\ref{alg:ScoringAlg} is based on Newton's method, its convergence depends heavily on the starting point, see also \cite{hojsgaard2012graphical}.

\subsubsection{Scoring algorithm for maximum dual log-likelihood estimation}

A simple Newton's method for solving the second step~(S2), that is the dual log-likelihood maximization problem under symmetry constraints, updates a starting point $\nu^0$ such that $\nu^{m+1}=\nu^m+\widecheck{I}(\nu^m)^{-1}\widecheck{S}(\nu^m)$.
Similar to the RCON algorithm above, we add a term $\widecheck{S}(\nu^m)\widecheck{S}(\nu^m)^T$ for stability, resulting in an iteration step
\begin{align}    
\nu^{m+1}=\nu^m+\left(\widecheck{I}(\nu^m)+\widecheck{S}(\nu^m)\widecheck{S}(\nu^m)^T\right)^{-1}\widecheck{S}(\nu^m),\label{eq:scoring_step_recipr}
\end{align}
which gives Algorithm~\ref{alg:ScoringAlg_recipr}. An implementation of this algorithm in \texttt{R} is available in the supplementary material.

\begin{algorithm} 
	\label{alg:ScoringAlg_recipr}
	$ $ \\[1ex]
	\textbf{Input:} {A step~(S1) solution $\widehat{Q}$, a graph $ G $, an edge coloring and a start value $\nu^0$.}
	
	\noindent
	\textbf{Initialize:} Compute the initial dual scores $\widecheck{S}(\nu^0)$.
	
	\noindent
	\textbf{Computation:} Iterate \eqref{eq:scoring_step_recipr} until the absolute sum of the dual scores is below a given threshold.
	
	\noindent
	\textbf{Output:} A parameter value $\widecheck{\nu}$.
\end{algorithm}

\subsection{Surrogate likelihood inference for H\"usler--Reiss RCON and RVAR models}
As discussed at the beginning of this section, the methodology for signed Laplacian-constrained Gaussian graphical models can be adapted for H\"usler--Reiss models. 
Assume independent observations $\y_1,\ldots,\y_n$ of a H\"usler--Reiss random vector $\Y$.
For a given RCON model, we are interested in finding the maximum of $\ell\left(\omega,\mathbf{t}(\widetilde{\Gamma})\right)$
where $\mathbf{t}(\widetilde{\Gamma})$ is a linear function of the sample variogram $\widetilde{\Gamma}$ of the corresponding $\W$, see \eqref{eq:summary_stat}.
Equivalently, for a given RVAR model the summary statistic in the first step~(S1) is a linear function of the sample variogram.
Because we can not directly observe $\W$, we cannot compute $\widetilde{\Gamma}$ directly from $\y_1,\ldots,\y_n$. Therefore we approximate the sample variogram using the empirical variogram of \cite{EV2020}.

Define index sets $ \mathcal{I}_k=\{i\in[n]:y_{ik}>0\} $ that describe all observations where the $ k $th entry exceeds zero. 
When $|\mathcal{I}_k|\ge 2$, let $ (\z_i)_{i\in\mathcal{I}_k} $ be the corresponding subset of centered observations, i.e.~$ \z_i:=\y_{i}-\frac{1}{|\mathcal{I}_k|}\sum_{\ell\in\mathcal{I}_k}\y_{\ell} $.
This yields the sample variogram  
\[\overline{\Gamma}^{(k)}_{ij}=\frac{1}{|\mathcal{I}_k|}\sum_{\ell\in\mathcal{I}_k} (z_{\ell i} - z_{\ell j})^2 \]
of the $ k $th extremal function $ \W^k $.
For $|\mathcal{I}_k|< 2$, we set $ \overline{\Gamma}^{(k)}=\mathbf{0} $.
The empirical variogram of \cite{EV2020} therefore equals
\[\overline{\Gamma}=\frac{1}{d}\sum_{k=1}^{d}\overline{\Gamma}^{(k)},\]
assuming that $|\mathcal{I}_k|\ge 2$ for at least one $ k\in [d] $.
We define the surrogate MLE of a H\"usler--Reiss RCON model given an empirical variogram $\overline{\Gamma}$ as
the solution $\widehat{\omega}$ of \eqref{eq:mlecol} given the sufficient statistic $\mathbf{t}(\overline{\Gamma})$.

Equivalently, we define a surrogate mixed dual estimator for a H\"usler--Reiss RVAR model given an empirical variogram $\overline{\Gamma}$ as the solution of the mixed dual estimator from Section~\ref{sec:learn_RVAR}. Hereby we use the empirical variogram $\overline{\Gamma}$ as sufficient statistic in the optimization problem~\eqref{eq:step1} for step~(S1), and solve the step~(S2) problem to obtain a surrogate mixed dual estimate $\widehat{\nu}$.

\section{Application}\label{sec:flights}
We demonstrate the application of the RCON and RVAR models on the flight delays data set of \citet{HES2022}, which is available in the \texttt{R} package \texttt{graphicalExtremes} \citep{graphicalExtremes}.
\subsection{Data}
The data set consists of total daily delays in minutes for airports in the United States. From this data set, we select all daily observations for airports with at least 2000 incoming and outgoing annual flights between 2005 and 2020.
We remove any day with missing observations, which results in $n=5347$ daily observations for $d=79$ airports.
The data is split into a training data set that contains the observations between 2005 and 2010 and a validation data set for observations between 2011 to 2020.

\subsection{Clustering}
Following \citet{HES2022} we separate the airports into smaller groups where simultaneous extreme events can be expected, for example due to extreme weather phenomena.
The airports are clustered using the $k$-medoids clustering algorithm \texttt{PAM} \citep{KR1990} that is implemented in the \texttt{R} package \texttt{cluster} \citep{cluster}.
Using the empirical variogram for a probability threshold of $p=0.85$ as dissimilarity measure, we split the airports into four clusters that are fairly geographically disjoint.
This yields a southern cluster with 12 airports, a western cluster (including Hawaii and Alaska) with 21 airports, a central cluster with 22 airports and an eastern cluster (including Puerto Rico) with 24 airports.
These are displayed in Figure~\ref{fig:clusters}, with on average at least monthly flight connections displayed as edges between the nodes.
\begin{figure}
    \centering
    \includegraphics[scale=0.7]{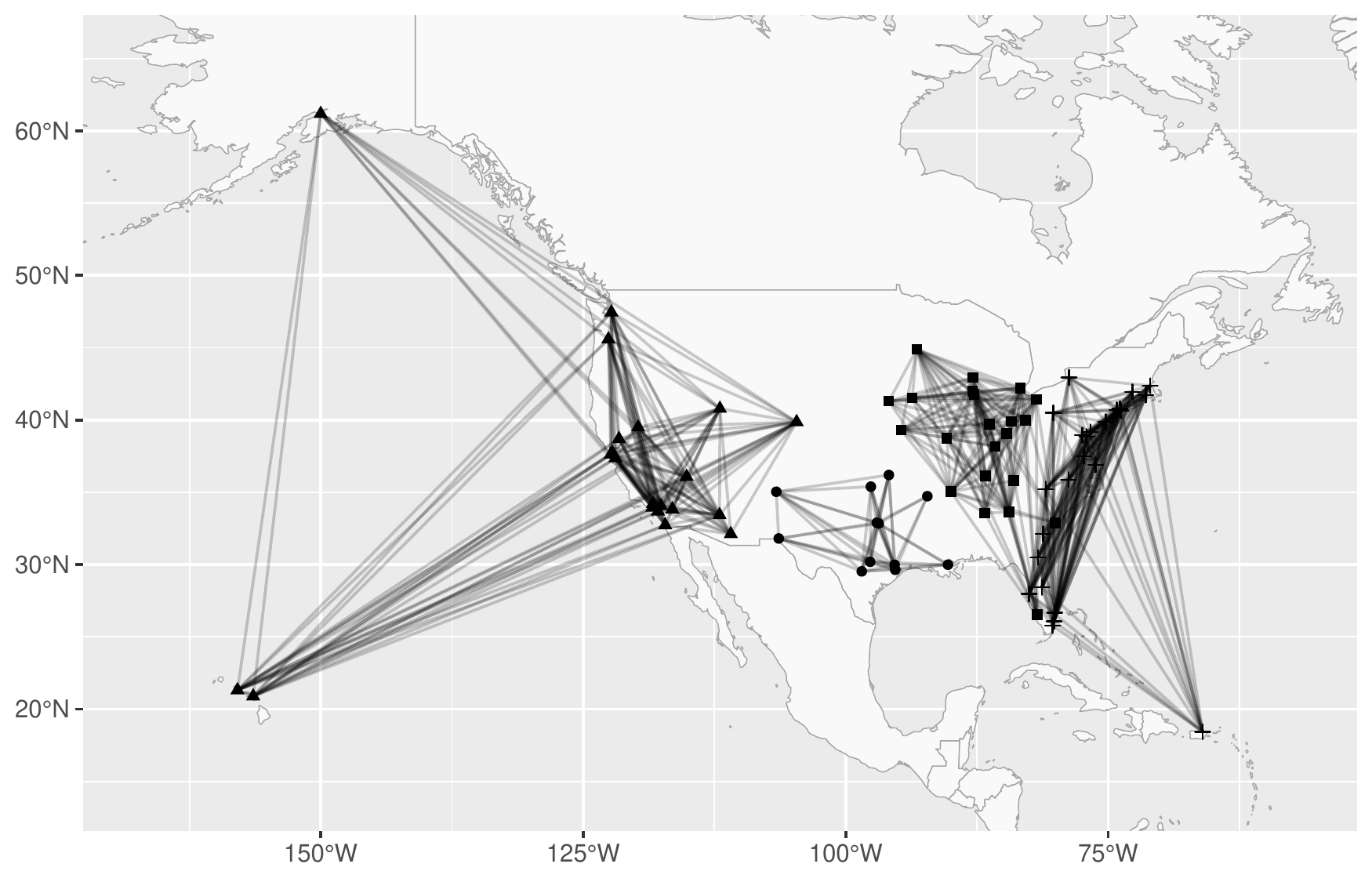}
    \caption{Flight connections for on average at least monthly flights within the four clusters.}
    \label{fig:clusters}
\end{figure}

For each cluster, we compute the empirical variogram for the training data set for a probability threshold of $p=0.85$.
Assuming a H\"usler--Reiss distribution, we further derive the surrogate maximum likelihood estimate under \EMTPtwo of \citet{REZ2021} using the \texttt{emtp2} algorithm available in the \texttt{R} package \texttt{graphicalExtremes} \citep{graphicalExtremes}.
This yields both a parameter and a graph estimate for this data set for each cluster. 
As an alternative structure learning approach, we test the \texttt{eglearn} structure learning algorithm of \citet{ELV2021}. 
We evaluate both \texttt{eglearn} and \texttt{emtp2} on the validation data set and observe that the latter performs better on every cluster, even without requiring a tuning parameter. 
Therefore we will consider the \texttt{emtp2} graph estimates as the underlying graphical model for RCON and RVAR inference.

\subsection{RCON inference}
For each cluster, our interest is in parameter reduction via an RCON model while maintaining a good fit on the validation data set.
For coloring the graph estimates, we apply the $k$-medoids algorithm from the \texttt{R} package \texttt{cluster} \citep{cluster}.
Using the empirical H\"usler--Reiss precision matrix $\overline{\Theta}$ on the edges of the graph estimates as dissimilarity weights, we separate the edges into $k=1,\ldots,25$ edge color classes.
For each $k$ we find the surrogate maximum likelihood estimate for the corresponding colored graphical models using Algorithm~\ref{alg:ScoringAlg}.
The resulting log-likelihoods on the validation data set for each coloring are presented in Figure~\ref{fig:likelihoods} together with the likelihoods of the original \EMTPtwo estimate and the emprical variogram.
\begin{figure}
    \centering
\begin{subfigure}{0.45\textwidth}
         \centering
         \includegraphics[scale=0.4]{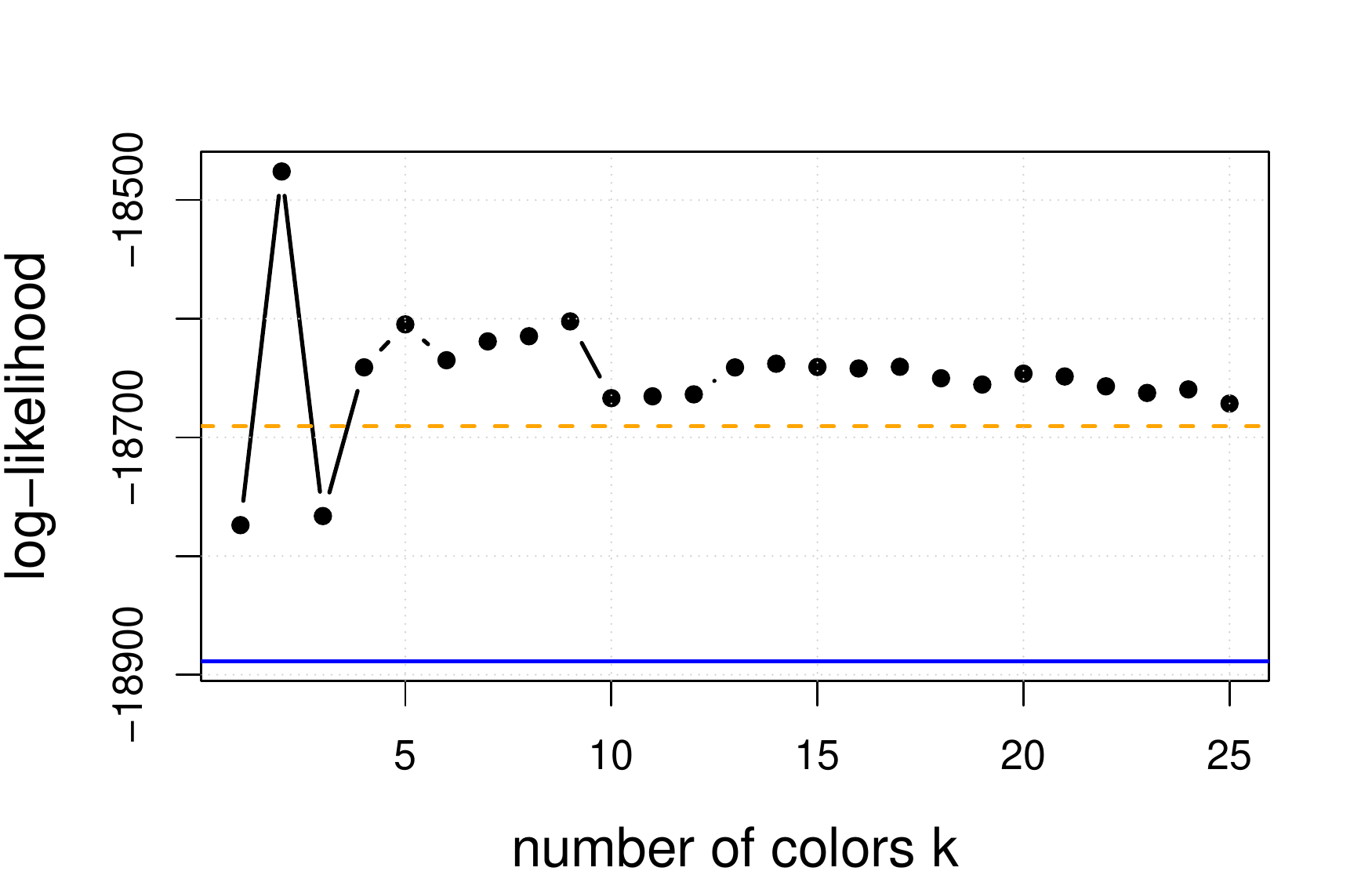}
         \caption{Southern cluster likelihoods}
 \end{subfigure}   
 \begin{subfigure}{0.45\textwidth}
         \centering
         \includegraphics[scale=0.4]{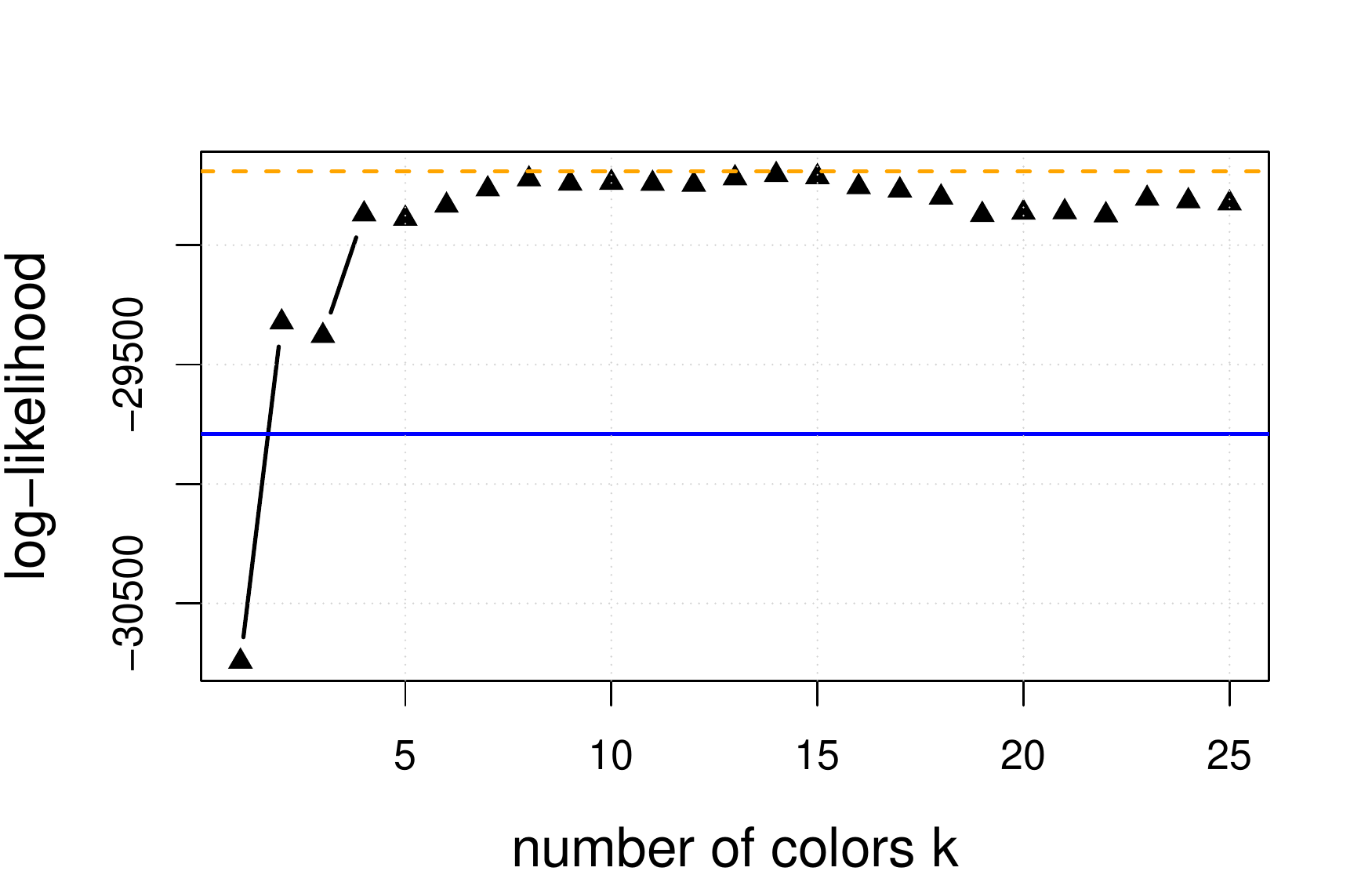}
         \caption{Western cluster likelihoods}
 \end{subfigure}    
 \begin{subfigure}{0.45\textwidth}
         \centering
         \includegraphics[scale=0.4]{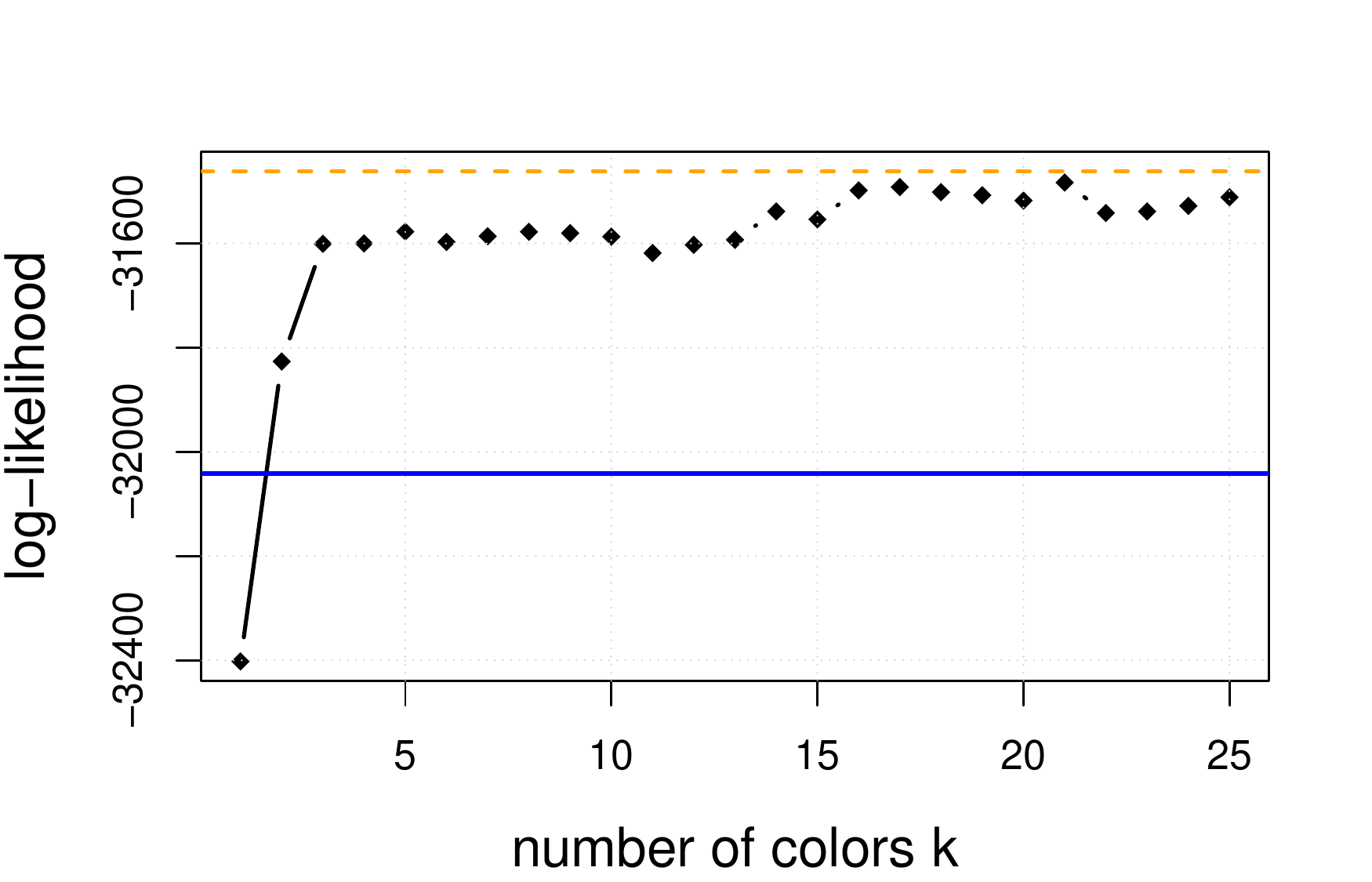}
         \caption{Central cluster likelihoods}
 \end{subfigure}   
 \begin{subfigure}{0.45\textwidth}
         \centering
         \includegraphics[scale=0.4]{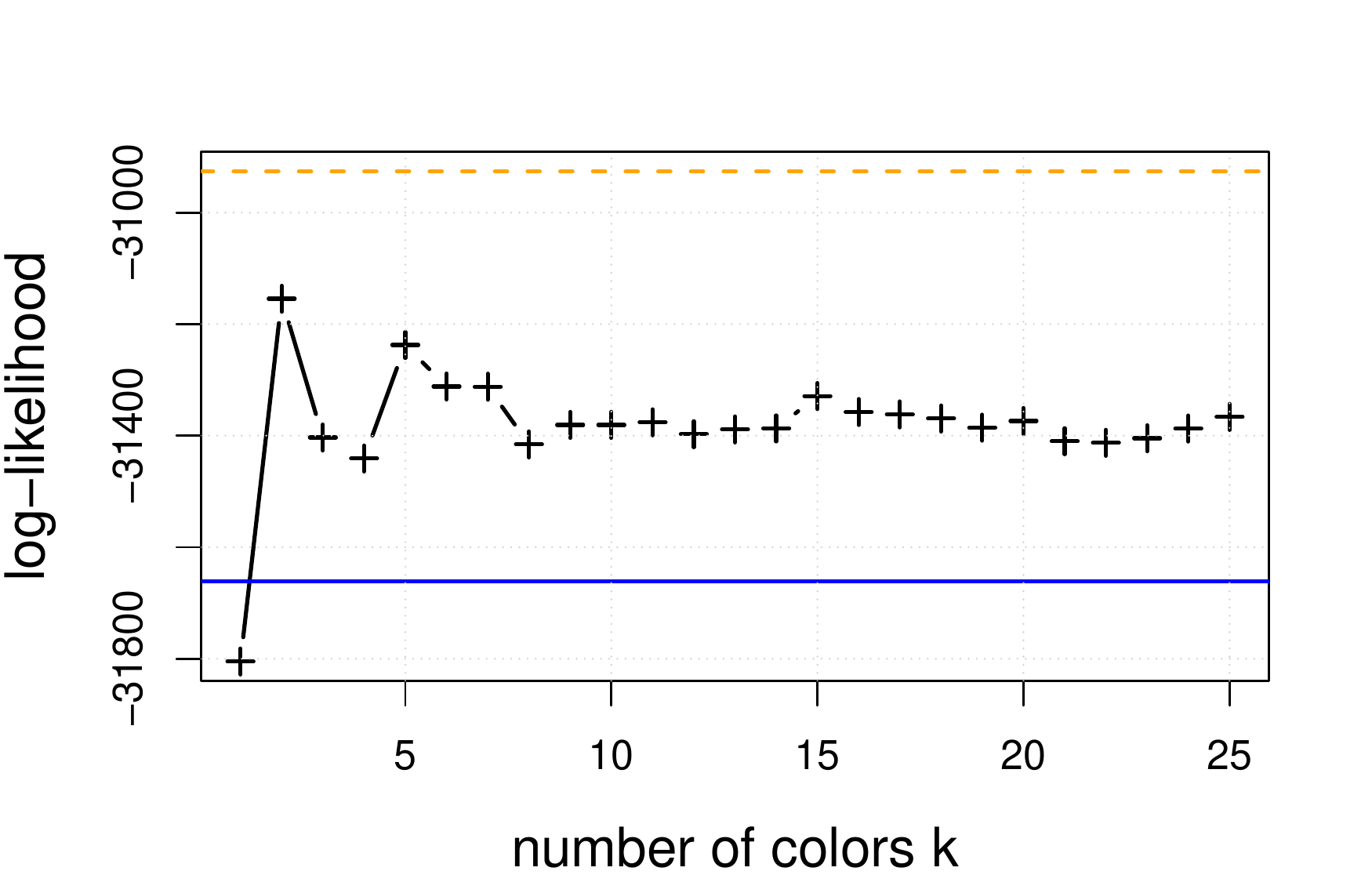}
         \caption{Eastern cluster likelihoods}
 \end{subfigure}    
     \caption{Log-likelihoods on the validation data for the surrogate maximum likelihood estimates for the  H\"usler--Reiss RCON models. The dashed line shows the log-likelihood for the \EMTPtwo estimate, and the solid line the log-likelihood of the empirical variogram.}
    \label{fig:likelihoods}
\end{figure}
We observe that for the central and western clusters the colored graphical models perform similarly to the \EMTPtwo estimate, while for the southern cluster, the colored estimate clearly outperforms the \EMTPtwo estimate even for a small number of colors. In comparison, for the eastern cluster the performance of the colored model is lower.
Comparing the best RCON model with respect to the validation data set with the \EMTPtwo estimates, we observe a large reduction of the number of parameters, see Table~\ref{tab:parameters}.
Figure~\ref{fig:colored_graphs} displays the colored graphs with the highest log-likelihood on the validation data for each cluster.
\begin{figure}
    \centering
\begin{subfigure}{0.45\textwidth}
         \centering
         \includegraphics[scale=0.35]{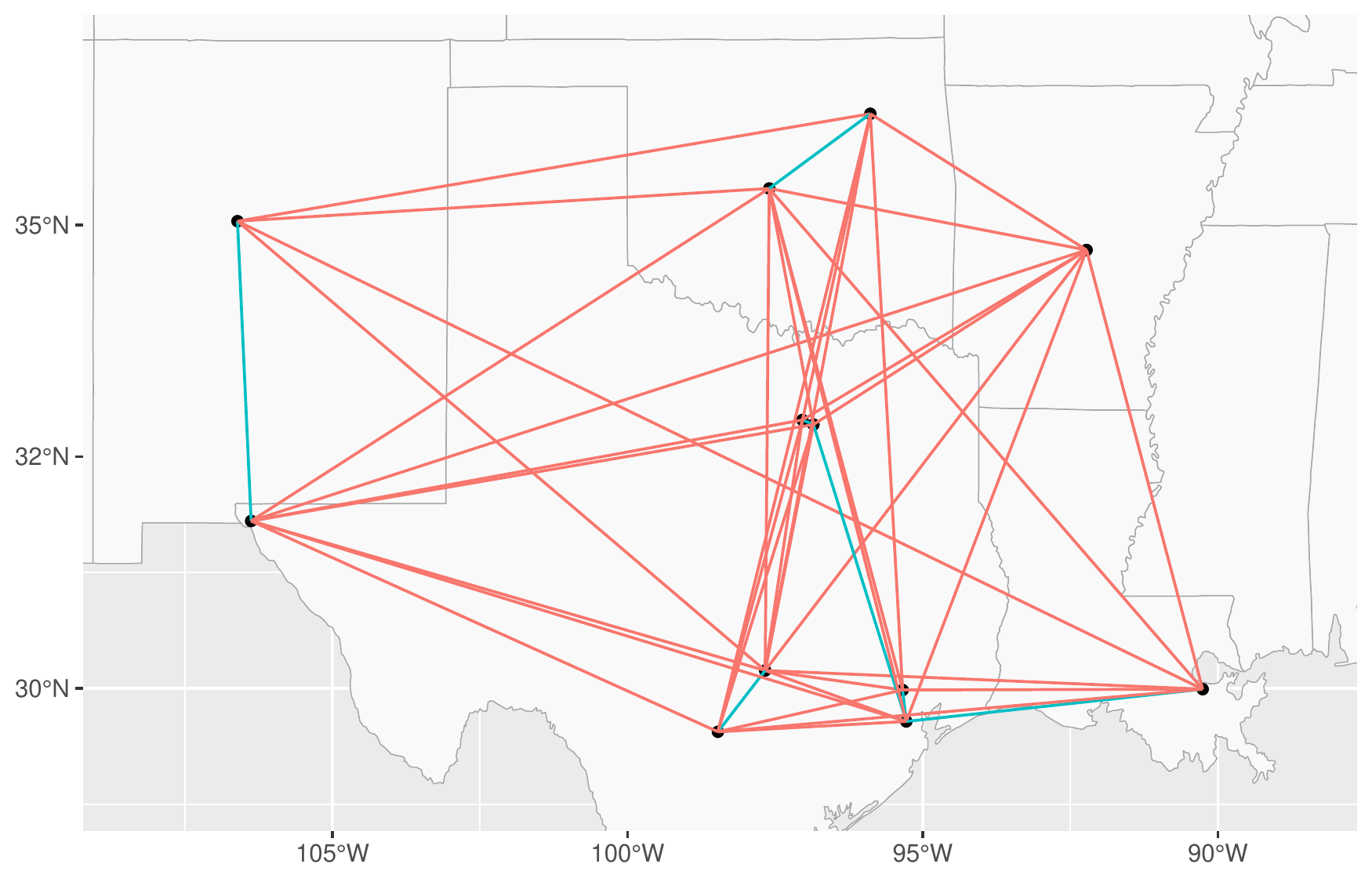}
         \caption{Southern cluster}
 \end{subfigure}   
 \begin{subfigure}{0.45\textwidth}
         \centering
         \includegraphics[scale=0.35]{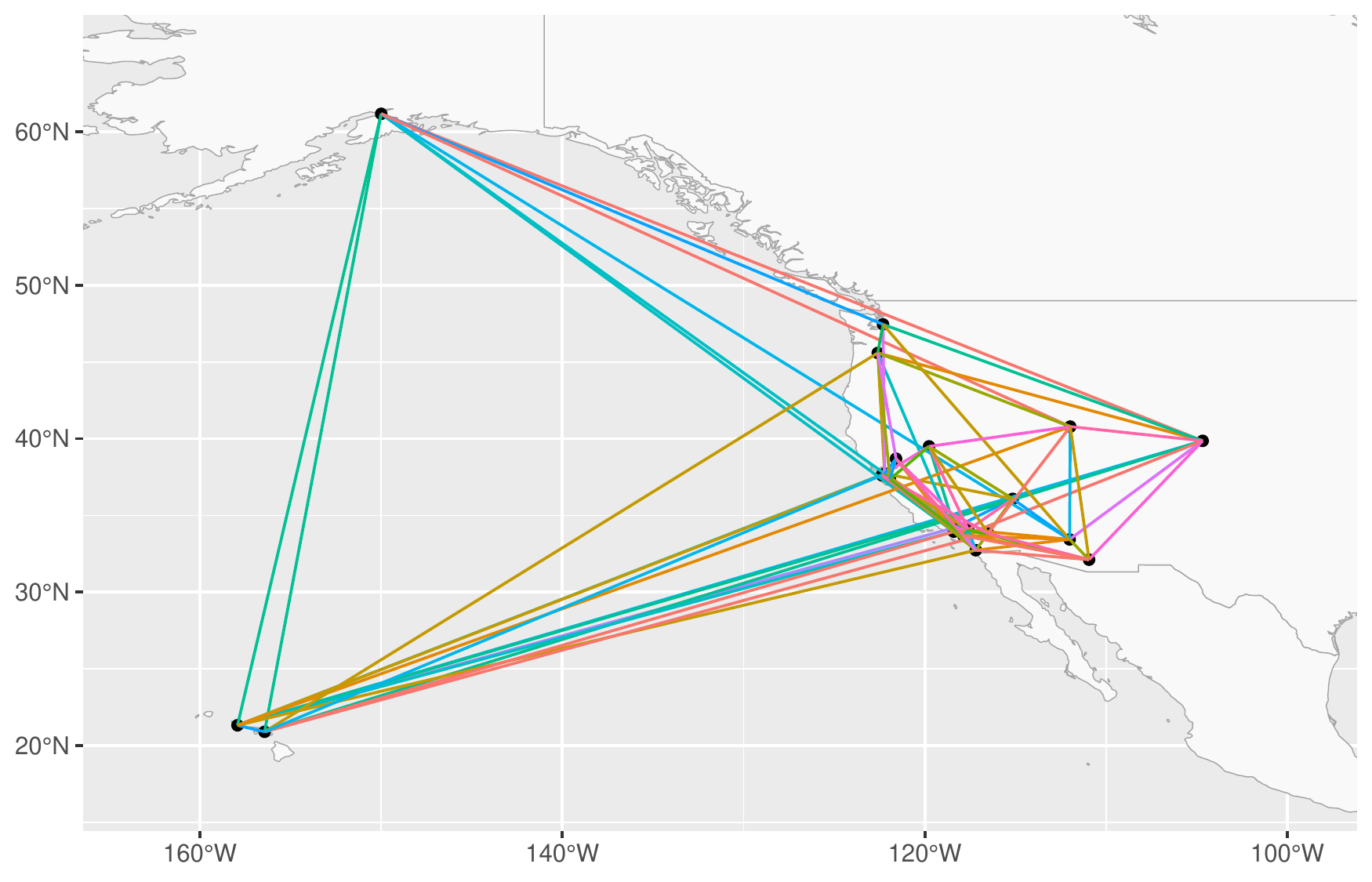}
         \caption{Western cluster}
 \end{subfigure}    
 \begin{subfigure}{0.45\textwidth}
         \centering
         \includegraphics[scale=0.35]{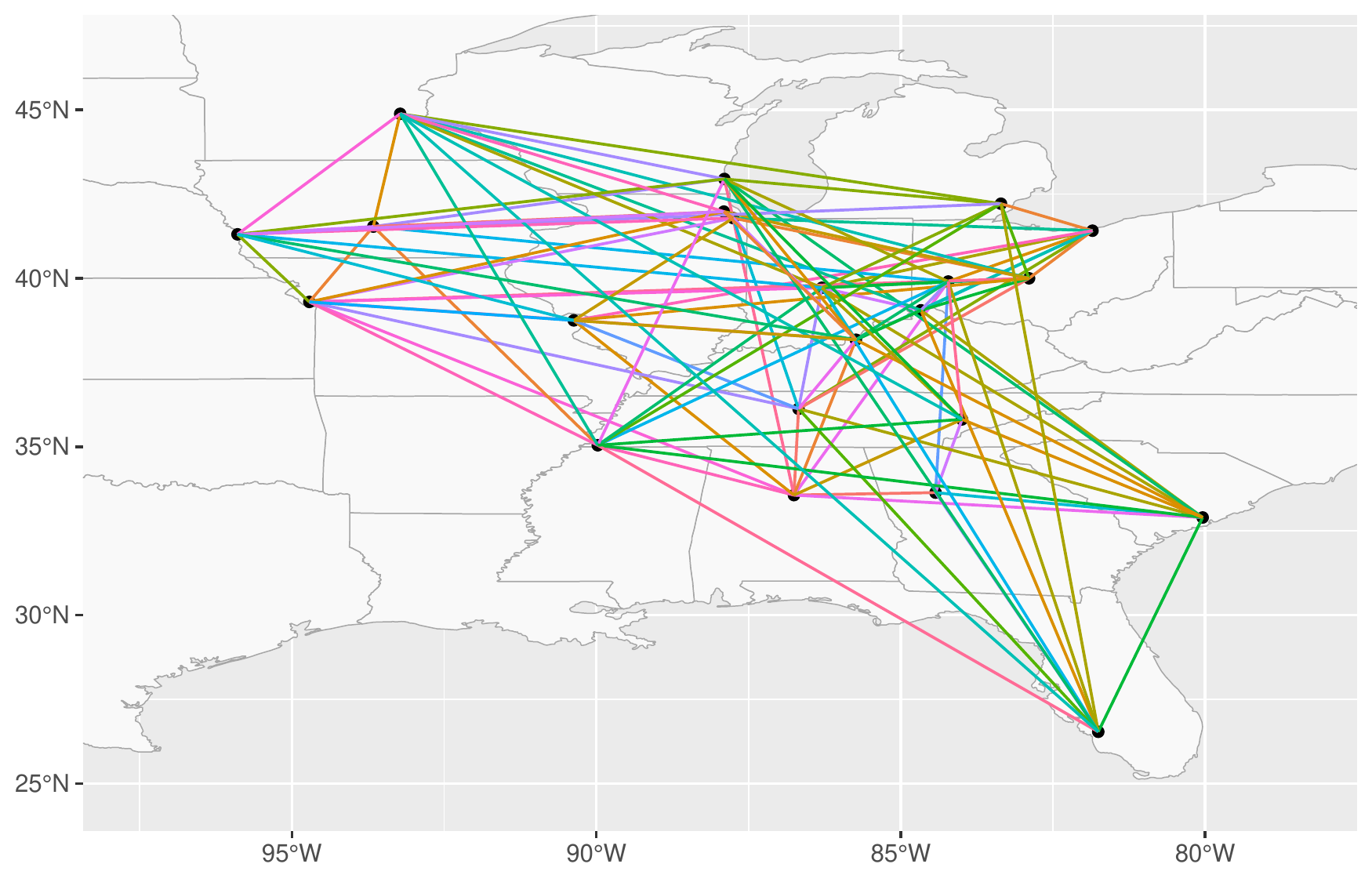}
         \caption{Central cluster}
 \end{subfigure}   
 \begin{subfigure}{0.45\textwidth}
         \centering
         \includegraphics[scale=0.35]{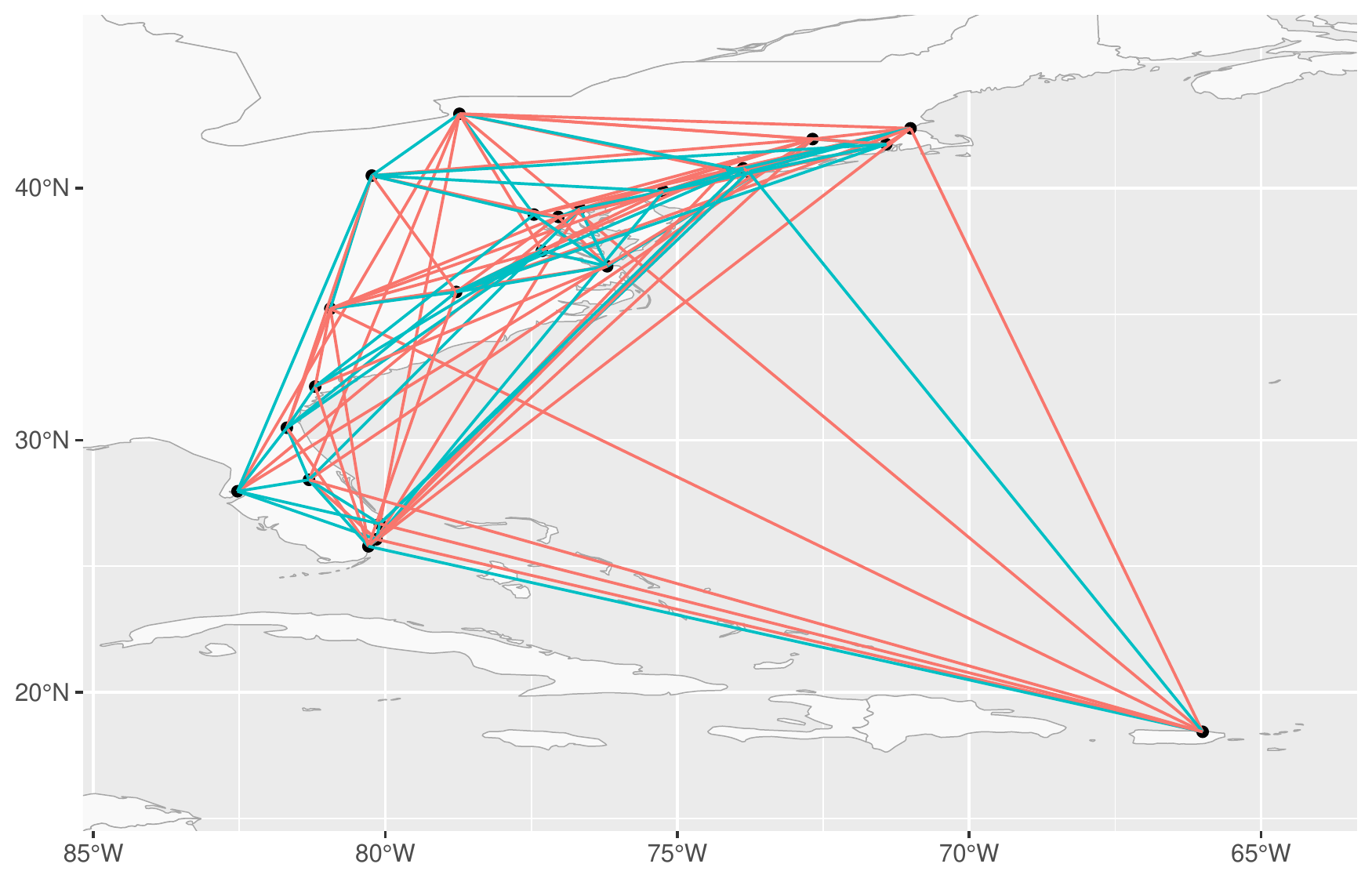}
         \caption{Eastern cluster}
 \end{subfigure}    
     \caption{Colored graphs with highest likelihood on the validation data for each cluster}
    \label{fig:colored_graphs}
\end{figure}
While the best colored graphs for the western and southern cluster have 14 and 22 different edge color classes, we observe in Figure~\ref{fig:likelihoods} that for both clusters we can have decent performance for small $k$.
The colored graph for the western cluster with $k=4$ and the central cluster with $k=3$ are shown in Figure~\ref{fig:colored_graphs2}.
\begin{figure}
    \centering
\begin{subfigure}{0.45\textwidth}
         \centering
         \includegraphics[scale=0.35]{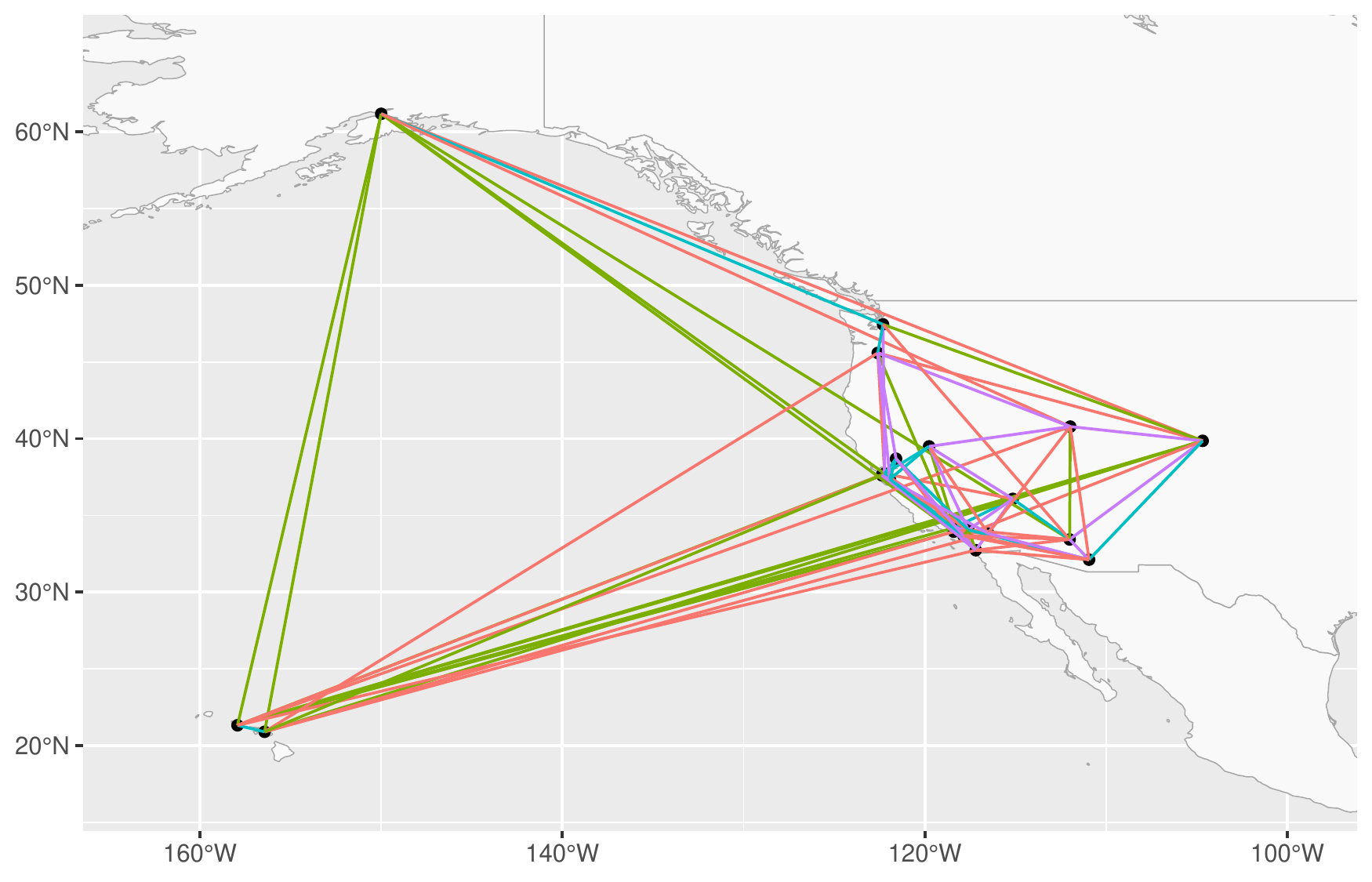}
         \caption{Western cluster ($k=4$)}
 \end{subfigure}   
 \begin{subfigure}{0.45\textwidth}
         \centering
         \includegraphics[scale=0.35]{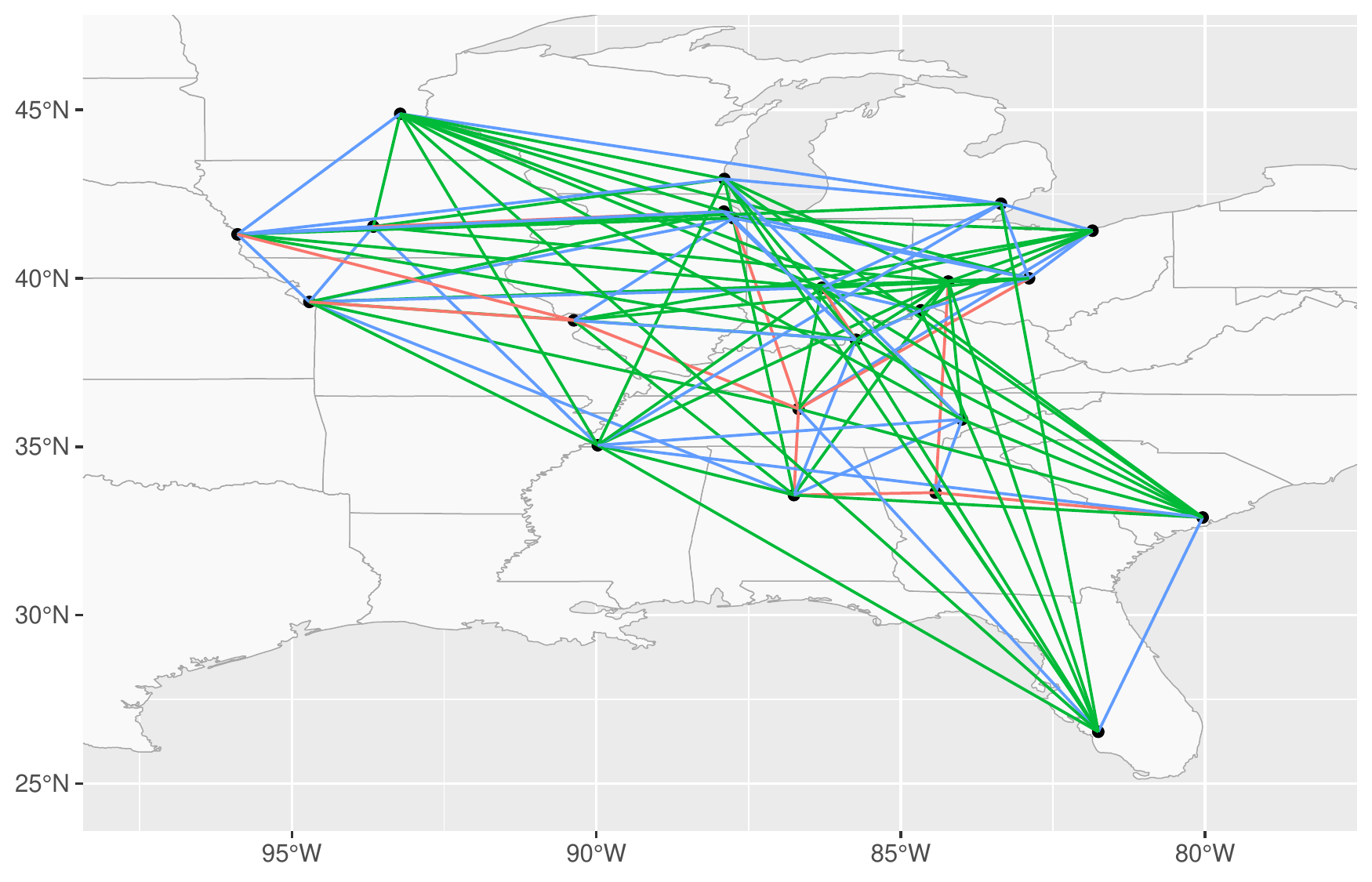}
         \caption{Central cluster($k=3$)}
 \end{subfigure}    
     \caption{Western and central cluster for small $k$ with reasonable performance.}
    \label{fig:colored_graphs2}
\end{figure}

\begin{table}[ht]
\centering
\begin{tabular}{rrrrr}
  \hline
 & Airports & nb par EMTP2 & nb par best RCON & nb par best RVAR \\ 
  \hline
Southern & 12 & 45 & 2 & 5\\ 
  Western & 21 & 91 & 14 & 10\\ 
  Central & 22 & 105 & 22 & 1\\ 
  Eastern & 24 & 101 & 2 & 3\\ 
   \hline
\end{tabular}
\caption{Parameter reduction for best colored models}\label{tab:parameters}
\end{table}

\subsection{RVAR inference}

We test the RVAR model for the flight delays data, where we use the \EMTPtwo estimator as graph estimate for each cluster.
Note that the \EMTPtwo estimator equals the surrogate MLE for the estimated graph, see \citet{REZ2021}.
This automatically yields the step~(S1) estimate $\widehat{Q}$ in \eqref{eq:step1}, as this is simply the surrogate MLE.
To assign the edge coloring, we apply a $k$-medoids clustering algorithm using the empirical variogram as dissimilarity measure for $k=1,\ldots,25$ for each cluster.
For the second step~(S2) as in \eqref{eq:step2}, for each edge coloring we compute a parameter estimate $\widehat{\nu}$ using our implementation in \texttt{R} of the reciprocal scoring algorithm (see Algorithm~\ref{alg:ScoringAlg_recipr}).
The resulting estimates are evaluated on the validation data set via the H\"usler--Reiss log-likelihood.
Figure~\ref{fig:likelihoods_RVAR} shows the H\"usler--Reiss log-likelihood on the validation data set for each RVAR estimate for each cluster compared to the log-likelihoods of the \EMTPtwo estimate and the empirical variogram.
We note that as the \EMTPtwo estimate is also the step~(S1) estimate, this particurlaly indicates the effects of the second step~(S2).

\begin{figure}
    \centering
\begin{subfigure}{0.45\textwidth}
         \centering
         \includegraphics[scale=0.4]{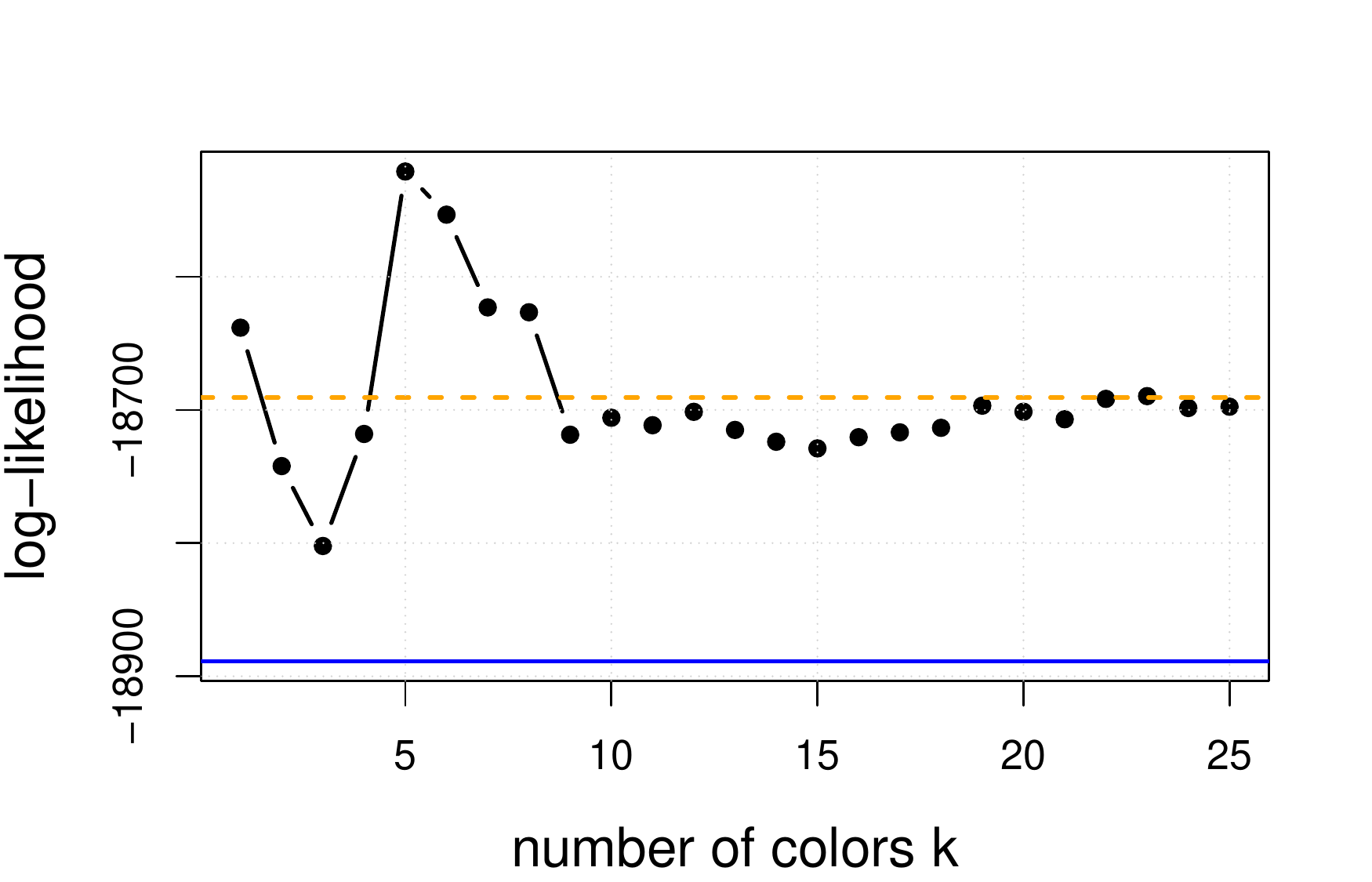}
         \caption{Southern cluster likelihoods}
 \end{subfigure}   
 \begin{subfigure}{0.45\textwidth}
         \centering
         \includegraphics[scale=0.4]{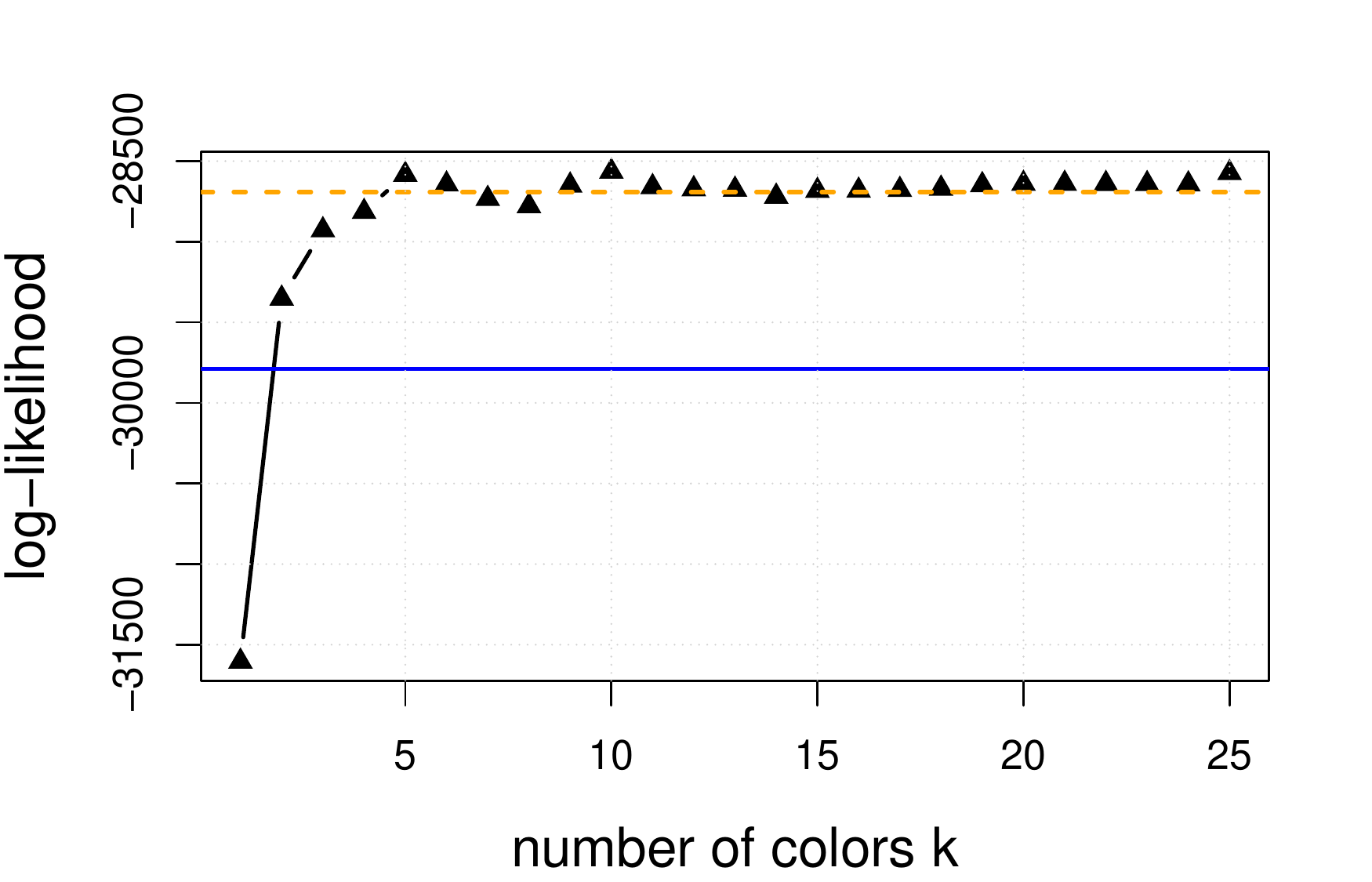}
         \caption{Western cluster likelihoods}
 \end{subfigure}    
 \begin{subfigure}{0.45\textwidth}
         \centering
         \includegraphics[scale=0.4]{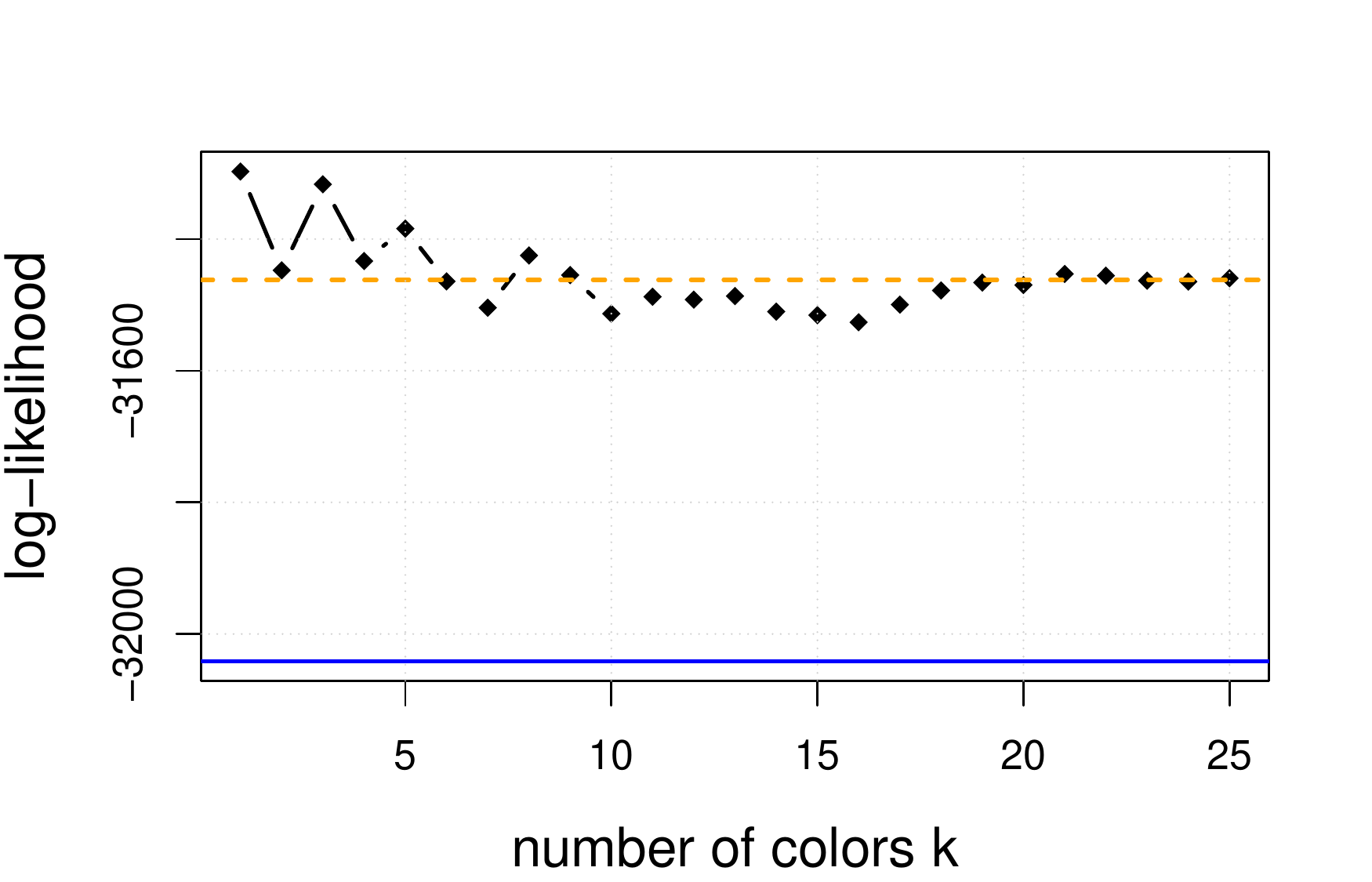}
         \caption{Central cluster likelihoods}
 \end{subfigure}   
 \begin{subfigure}{0.45\textwidth}
         \centering
         \includegraphics[scale=0.4]{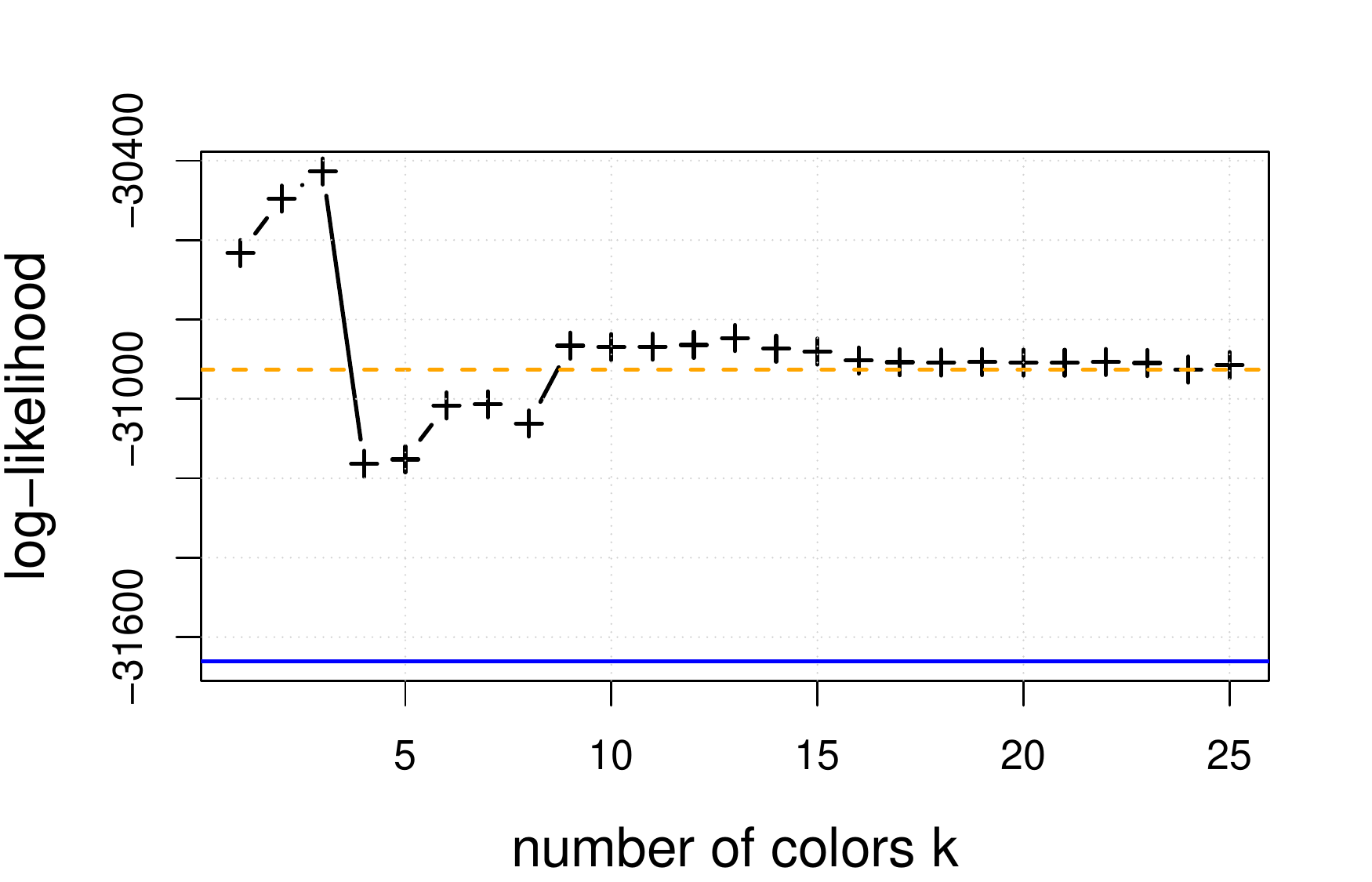}
         \caption{Eastern cluster likelihoods}
 \end{subfigure}    
     \caption{Log-likelihoods on the validation data for the surrogate maximum likelihood estimates for the RVAR H\"usler--Reiss graphical models. The dashed line shows the log-likelihood for the \EMTPtwo estimate, and the solid line the log-likelihood of the empirical variogram.}
    \label{fig:likelihoods_RVAR}
\end{figure}
We observe that the best RVAR estimates clearly outperform the \EMTPtwo estimate for the southern, central and eastern cluster for very small $k$, see Figure~\ref{fig:likelihoods_RVAR} and Table~\ref{tab:parameters}. For the western cluster, we obtain comparable results to the \EMTPtwo model, but with a similarly strong parameter reduction.

Figure~\ref{fig:colored_graphs_RVAR} displays the RVAR graphs with the highest log-likelihood on the validation data for each cluster.
Interestingly, the RVAR model for the central cluster with the best performance is achieved for $k=1$, leading to a one-parametric model and a monochromatic graphical model. Furthermore, the RVAR model for the eastern cluster ($k=3$) shows an interesting geographical structure, where all edges connecting to Puerto Rico belong to the same edge color class.
A similar behavior is observed in Figure~\ref{fig:colored_graphs_RVAR5} for the western cluster for $k=5$ for flight connections to Hawaii. 
We interpret this as strong evidence for parameter symmetry with respect to flight connections between continental and island airports.
\begin{figure}
    \centering
\begin{subfigure}{0.45\textwidth}
         \centering
         \includegraphics[scale=0.35]{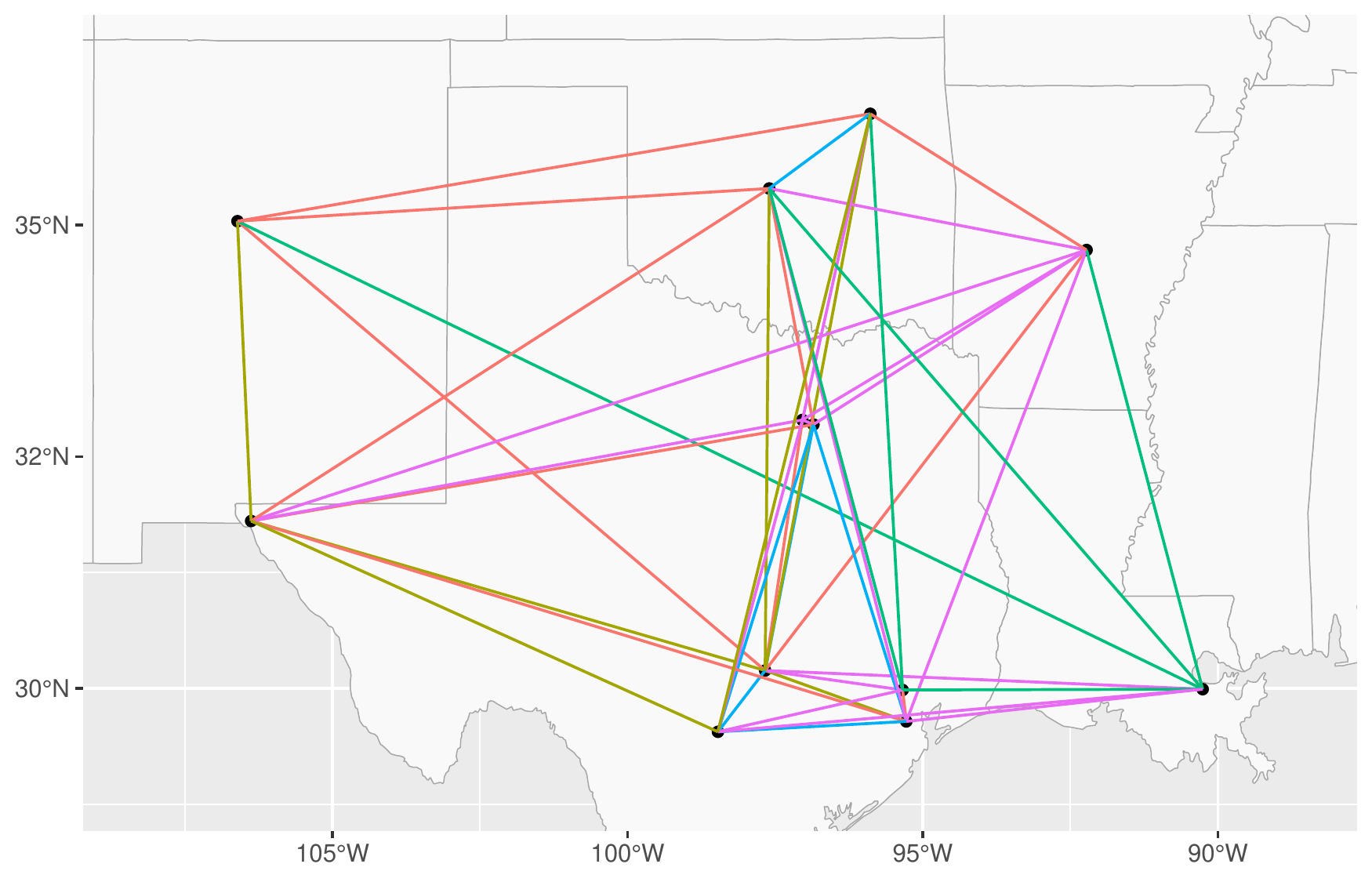}
         \caption{Southern cluster}
 \end{subfigure}   
 \begin{subfigure}{0.45\textwidth}
         \centering
         \includegraphics[scale=0.35]{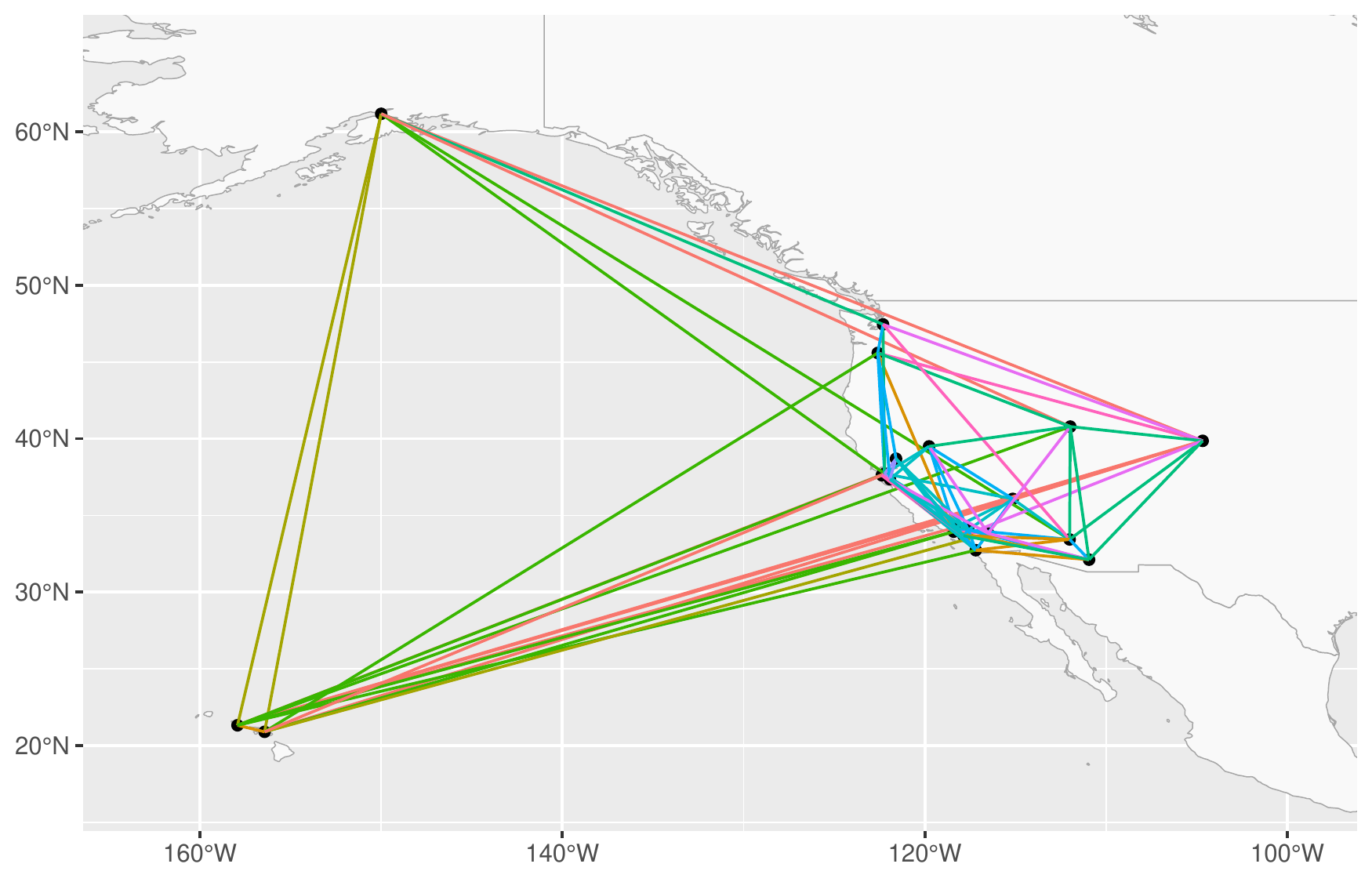}
         \caption{Western cluster}
 \end{subfigure}    
 \begin{subfigure}{0.45\textwidth}
         \centering
         \includegraphics[scale=0.35]{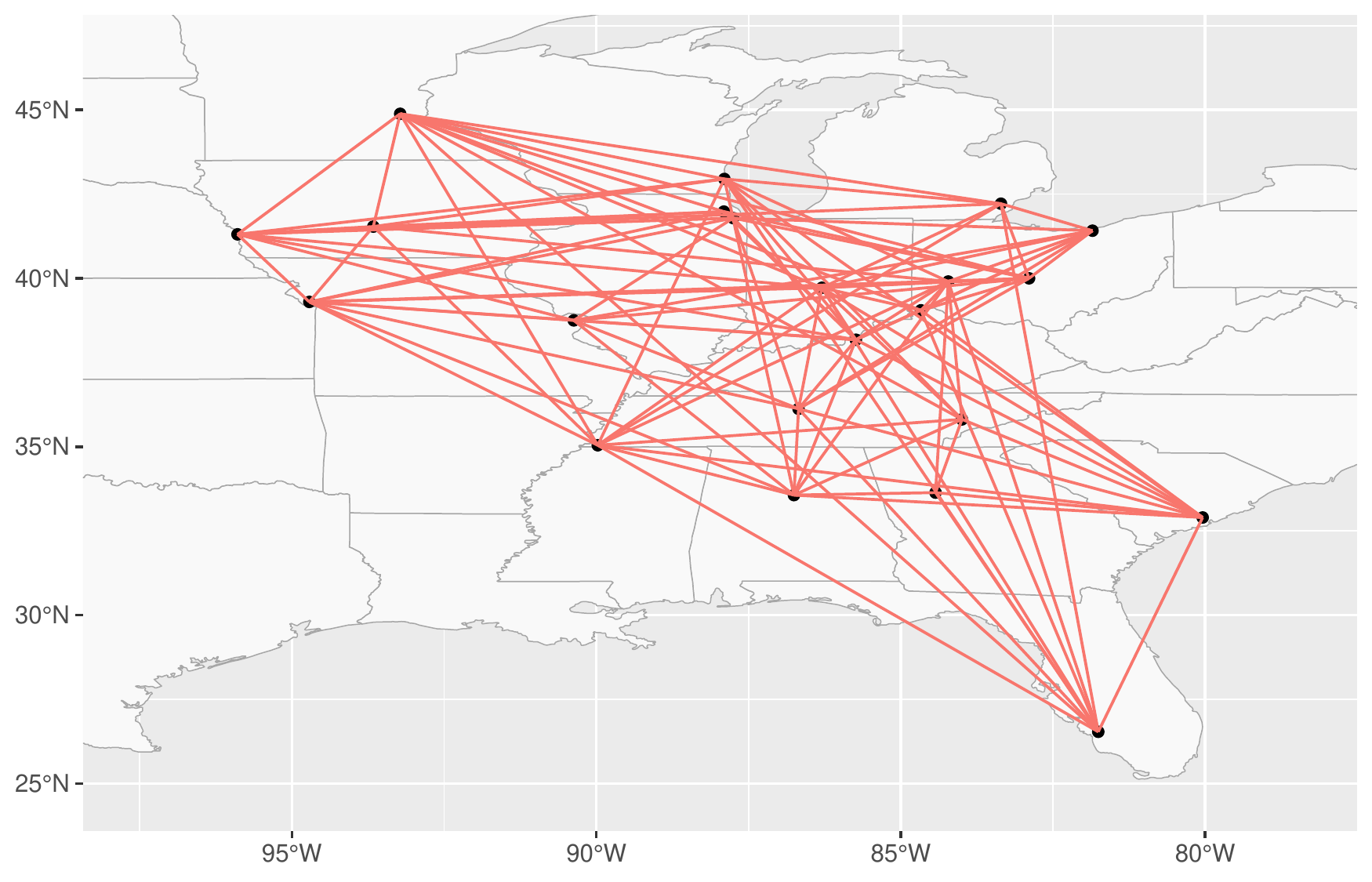}
         \caption{Central cluster}
 \end{subfigure}   
 \begin{subfigure}{0.45\textwidth}
         \centering
         \includegraphics[scale=0.35]{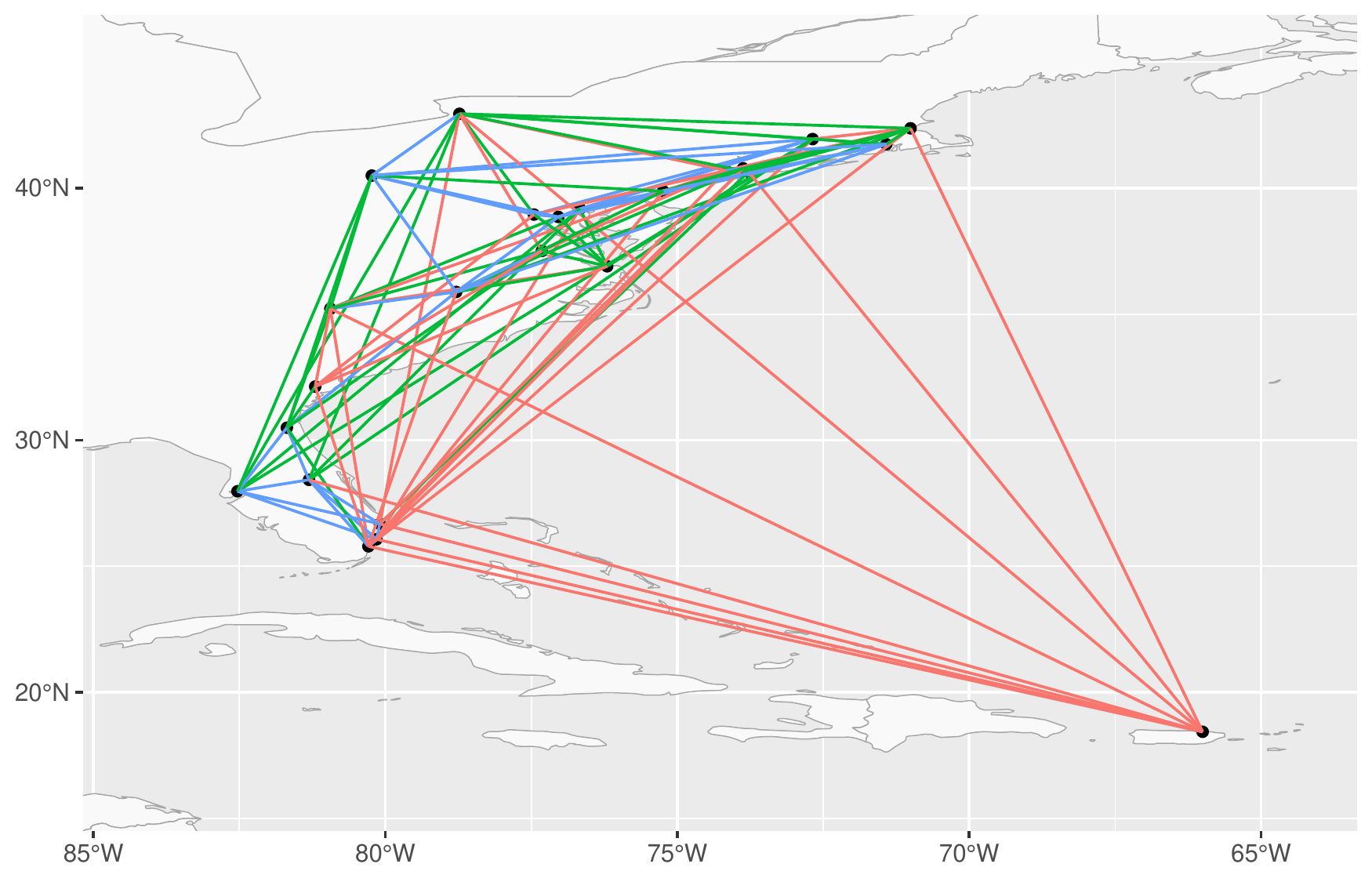}
         \caption{Eastern cluster}
 \end{subfigure}    
     \caption{RVAR graphs with highest likelihood on the validation data for each cluster}
    \label{fig:colored_graphs_RVAR}
\end{figure}

\begin{figure}
    \centering
 \begin{subfigure}{0.45\textwidth}
         \centering
         \includegraphics[scale=0.35]{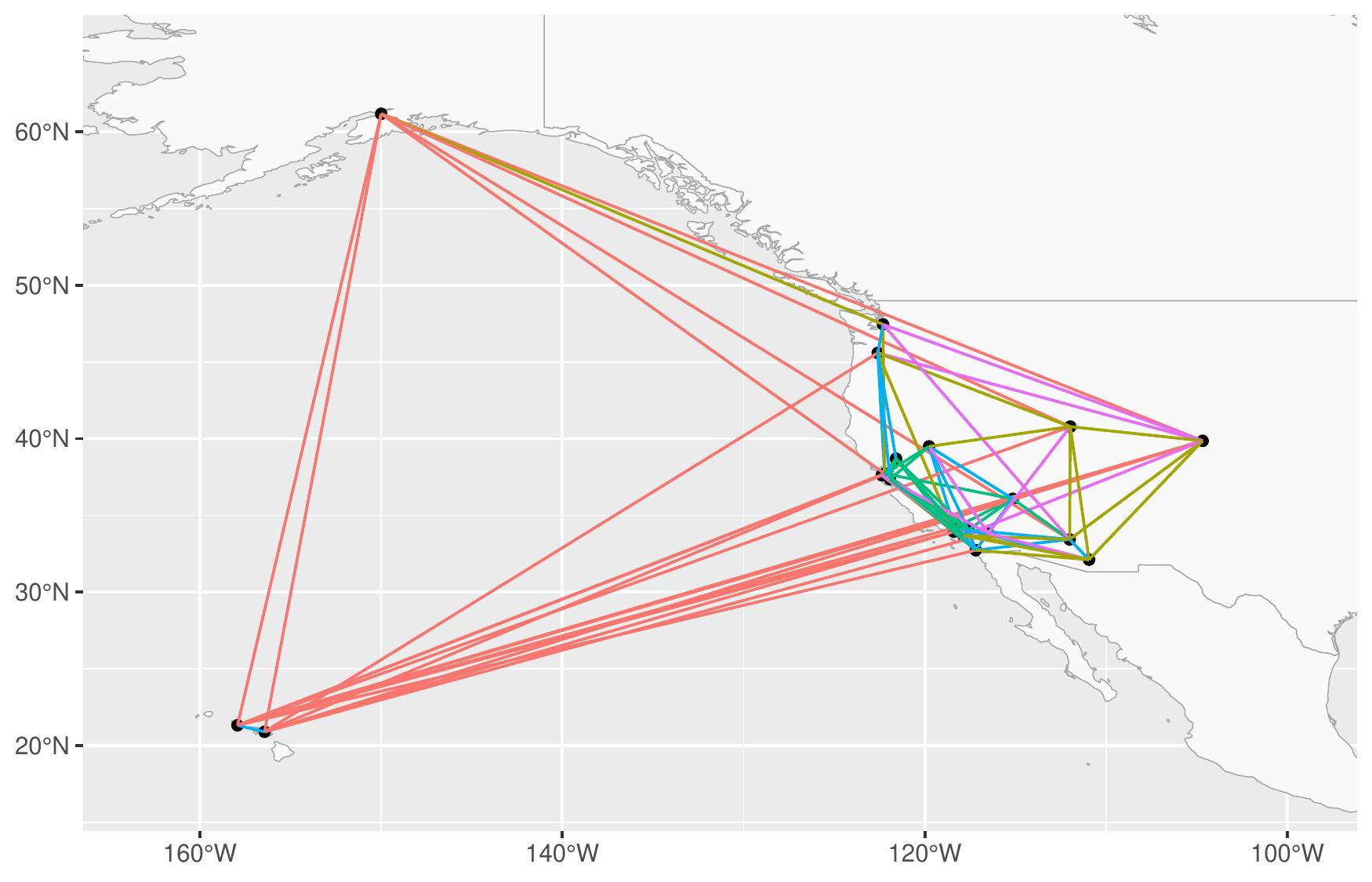}
    \end{subfigure} 
    \caption{Colored western cluster for $k=5$.}
    \label{fig:colored_graphs_RVAR5}
\end{figure}

In summary, we conclude that the RVAR model provides an excellent fit for all clusters, while providing for a strong parameter dimension reduction.

\subsection*{Acknowledgements}
Frank R\"ottger was supported by the Swiss National Science Foundation (Grant 186858).
The authors thank Sebastian Engelke and Manuel Hentschel for helpful discussions and technical support. 
Alexandros Grosdos was supported by the European Research Council (ERC) under the European Union’s Horizon 2020 research and innovation programme (Grant agreement No.~883818).
\appendix

\section{Proofs}
\subsection{Proof of Lemma~\ref{lem:infmat}}\label{prf:infmat}
The information matrix for exponential families is the negative of the Hessian matrix of the log-partition function, or as the negative of the Jacobian matrix of the score functions. We compute
\begin{align*}
    \frac{\partial}{\partial Q_{st}}\Gamma_{uv}&=-\frac{\partial\Sigma^{(v)}_{uu}}{\partial Q_{st}}= (\Sigma^{(v)}_{u})^T \frac{\partial\Theta^{(v)}}{\partial Q_{st}}\Sigma^{(v)}_{u}
\end{align*}
Now, it holds that
\[ \frac{\partial\Theta^{(v)}}{\partial Q_{st}}=\begin{cases} (e_s-e_t)(e_s-e_t)^T, &s,t\neq u,v,\\
e_s e_s^T, & t=v.
\end{cases} \]
Hence, we obtain
\begin{align*}\frac{\partial}{\partial Q_{st}}\Gamma_{uv}&=-\Sigma_{us}^{(v)}(\Sigma_{us}^{(v)}-\Sigma_{ut}^{(v)})+\Sigma_{ut}^{(v)}(\Sigma_{ut}^{(v)}-\Sigma_{us}^{(v)})\\
&=-(\Sigma_{us}^{(v)}-\Sigma_{ut}^{(v)})^2\\
&=-\frac{1}{4}(\Gamma_{sv}-\Gamma_{us}-\Gamma_{tv}+\Gamma_{ut})^2.\\
\end{align*}
This implies with the multivariable chain rule that
\begin{align*}
    I(\omega)_{ij}&= -\frac{\partial}{\partial \omega_j}S_i(\omega)\\
    &= -\frac{1}{2}\sum_{uv \in E: \;\lambda(uv)=i} \frac{\partial}{\partial \omega_j}\Gamma_{uv}(\omega)\\
    &=  -\frac{1}{2}\sum_{uv \in E: \;\lambda(uv)=i}\sum_{st \in E: \;\lambda(st)=j} \frac{\partial}{\partial Q_{st}}\Gamma_{uv}(\omega)\\
    &=\frac{1}{8}\sum_{uv \in E: \;\lambda(uv)=i}\sum_{st \in E: \;\lambda(st)=j} (\Gamma_{sv}(\omega)-\Gamma_{us}(\omega)-\Gamma_{tv}(\omega)+\Gamma_{ut}(\omega))^2,
\end{align*} 
which proves the claim. \qed

\subsection{Proof of Lemma~\ref{lem:infmat_recipr}}\label{prf:infmat_recipr}

We compute using Lemma~\ref{lem:CMinv} that
\begin{align*}
    \frac{\partial}{\partial \Gamma_{st}}Q_{uv} &= -\frac{\partial }{\partial \Gamma_{st}}\Theta_{uv}\\
    &=(\Theta_{u1},\ldots,\Theta_{ud},r_u)\frac{\partial \CM(\Gamma)}{\partial \Gamma_{st}}\begin{pmatrix}
        \Theta_{v1}\\
        \vdots\\
        \Theta_{vd}\\
        r_v\\
    \end{pmatrix}\\
    &=-\frac{1}{2} (\Theta_{ut}\Theta_{vs}+\Theta_{us}\Theta_{vt}).
\end{align*}
We therefore obtain
\begin{align*}
    \widecheck{I}(\nu)_{ij}&=-\frac{\partial}{\partial \nu_j}\widecheck{S}_i(\nu)\\
    &=-\frac{1}{2}\sum_{uv \in E: \;\lambda(uv)=i} \frac{\partial}{\partial \nu_j}Q_{uv}(\nu)\\
    &=-\frac{1}{2}\sum_{uv \in E: \;\lambda(uv)=i}\sum_{st \in E: \;\lambda(st)=j}\frac{\partial}{\partial \Gamma_{st}}Q_{uv}(\nu)\\
    &=\frac{1}{4}\sum_{uv \in E: \;\lambda(uv)=i}\sum_{st \in E: \;\lambda(st)=j}(\Theta_{ut}(\nu)\Theta_{vs}(\nu)+\Theta_{ut}(\nu)\Theta_{vs}(\nu)),
\end{align*}
as needed. \qed

\bibliographystyle{chicago}
\bibliography{bibliography}

\end{document}